\documentclass[preprint,12pt]{imsart}

\RequirePackage[OT1]{fontenc}
\RequirePackage{amsthm,amsmath}
\RequirePackage[numbers]{natbib}
\RequirePackage[colorlinks,citecolor=blue,urlcolor=blue]{hyperref}

\usepackage{amscd,amsfonts,amssymb,amsmath,latexsym,array,hhline,xcolor,graphicx}
\usepackage{float}
\usepackage{appendix}
\usepackage{booktabs}
\usepackage{caption}
\usepackage{bm}
\usepackage{epstopdf}

\newcommand\cF{{\cal F}}

\newcommand\cB{{\cal B}}
\newcommand\cH{{\cal H}}

\newcommand\cN{{\cal N}}

\newcommand\cR{{\cal R}}
\newcommand\cO{{\cal O}}
\newcommand\cP{{\cal P}}

\newcommand\R{{\bf R}}
\newcommand\Q{{\bf Q}}

\newcommand\G{{\bf G}}

\newcommand\N{\mathbb{N}}

\newcommand\C{{\rm C}}

\newtheorem{theo}{Theorem}[section]
\newtheorem{prop}[theo]{Proposition}
\newtheorem{lemm}[theo]{Lemma}
\newtheorem{defi}[theo]{Definition}
\newtheorem{ex}[theo]{Example}
\newtheorem{Assum}{Assumption}

\newtheorem{coro}[theo]{Corollary}
\newtheorem{rem}[theo]{Remark}
\newcommand\fdem{$\Box$}
\newcommand\beq{\begin{equation}}
\newcommand\eeq{\end{equation}}

\newcommand\bea{\begin{eqnarray}}
\newcommand\eea{\end{eqnarray}}
\newcommand\bean{\begin{eqnarray*}}
\newcommand\eean{\end{eqnarray*}}

\DeclareMathOperator{\essinf}{ess\;inf}
\DeclareMathOperator{\esssup}{ess\;sup}

\newcommand\bP{{\bf {\rm P}}}

\begin{document}

\begin{frontmatter}

\title{Dynamic programming principle and computable prices in financial market models with transaction costs.}

\author[A1]{ Emmanuel LEPINETTE,}
\author[A2]{ Duc Thinh Vu}

\address[A1]{ Ceremade, UMR  CNRS 7534,  Paris Dauphine University, PSL National Research, Place du Mar\'echal De Lattre De Tassigny, \\
75775 Paris cedex 16, France and\\
Gosaef, Faculty of Sciences of Tunis, Tunisia.\\
Email: emmanuel.lepinette@ceremade.dauphine.fr
}
\address[A2]{ Ceremade, UMR  CNRS 7534,  Paris Dauphine University, PSL National Research, Place du Mar\'echal De Lattre De Tassigny, \\
75775 Paris cedex 16, France.\\
Email: vu@ceremade.dauphine.fr
}

\begin{abstract} How to compute (super) hedging costs in rather general financial market models with transaction costs in discrete-time ? Despite the huge literature on this topic, most of results are characterizations of the super-hedging prices while it remains difficult to deduce numerical procedure to estimate them. We establish here a dynamic programming principle and we prove that it is possible to implement it under some conditions on the conditional supports of the price and volume processes for a large class of market models including convex costs such as order books but also non convex costs, e.g. fixed cost models. \end{abstract}


\end{frontmatter}

\section{Introduction}
The problem of characterizing the set of all possible prices  hedging a European claim has been extensively studied in the literature under classical no-arbitrage conditions. In discrete-time and without transaction costs, a dual characterization is deduced through  dual elements, the equivalent martingale measures, whose existence characterizes the well known no-arbitrage condition NA, see the FTAP theorem of \cite{DMW}. In continuous time,  similar characterizations are obtained under the NFLVR condition of  Delbaen and Schachermayer \cite{DScha1},  \cite{DScha2}  for instance. The Black and Scholes model \cite{BS} is the canonical example of complete market in mathematical finance such that the equivalent probability measure is unique. The advantage of this simple model is that  hedging prices  are explicitly given. Unfortunately, for incomplete market models, it is difficult to establish numerical procedures to estimate the super-hedging prices from the dual characterization. This is why it is usual to specify a particular martingale measure, see \cite{Sch}, \cite{Frit} and \cite{JKM}. 

In the presence of transaction costs, the financial market is a priori incomplete and computing the infimum super-hedging prices remains a challenge. In the Kabanov model with transaction costs \cite{KS}, the main result is a dual characterization \cite{KS}[Theorem 3.3] through the so-called consistent price systems (CPS) that characterize various kinds of no-arbitrage conditions for these models, see \cite{KS}[Section 3.2]. Unfortunately, it is difficult to identify the consistent price systems and deduce a numerical estimation of the prices. A first attempt (and the only one) is proposed in \cite{LR} for finite probability spaces. More generally, vector optimization methods are proposed for risk measures as in \cite{CM} still for finite probability spaces. Also, various asymptotic results are obtained for small transaction costs by Schachermayer \cite{SCh}, \cite{GRS} and others \cite{Lep5}, \cite{PergL}, still for conic models. 

For non conic models, in the presence of an order book for instance, more generally with convex cost,  or with fixed costs, few results are available in the literature. Well known papers such as \cite{JN}, \cite{PP}, \cite{Pennanen} , \cite{LT}, \cite{LT1} only formulate characterizations of the super-hedging prices. The very question we aim to address in this paper is how to numerically  compute the infimum super-hedging cost of a European claim. 

To do so, we first provide a dynamic programming principle in a very general setting in discrete time, see Theorem \ref{DPP}. Notice that we do not need any no-arbitrage condition to formulate it. Secondly, we propose some conditions under which it is possible to implement the dynamic programming principle. Actually, we shall see that we only need to have an insight on the conditional supports of the increments of the process describing the financial market, mainly the price and volume processes. 

Our main results are formulated under some weak no-arbitrage conditions such that the minimal super-hedging costs are non negative for non negative payoffs, as in \cite{CL}, \cite{BCL}, \cite{EL1}. These conditions avoid the unrealistic case of infinitely negative prices. The main problem is how to compute an essential supremum and an essential infimum. We show that they may coincide with pointwise supremum and infimum respectively. This is sufficient  to compute backwardly the hedging costs as solutions to pointwise (random) optimization problems. 

The main difficulty is to find some conditions, a priori the weakest ones are the best, so that the dynamic programming principle reads as the simpler problem of computing pointwise supremum and infimum as announced above.  In this paper, we naturally suppose that the payoff we consider is (super) hedgeable as the contrary assumption does not make sense in the problem of estimating a price. Moreover, for computational purposes, we suppose that the prices can not be infinite. In practice, this is clearly observed and this leads to the condition AIP, which is weaker than the usual no-arbitrage conditions. 

We do not necessarily suppose that the transactions costs are convex but, when this is the case, we show that convexity is preserved backwardly, i.e. the infimum price of the terminal claim is a convex function of the current price at any time. In general, even with non convex transaction costs, we obtain computable prices under a condition imposed on the conditional supports of the underlying price process. Recall that any closed random set admits a Castaing representation, i.e. it admits a.s. a countable dense subset composed of countable measurable selections of itself. In our paper, we suppose that this Castaing representation is a function of the current price at any time. This can be interpreted as a Markov property. A priori, we may extend our results under the condition that the Castaing representation depends on the  path of the underlying price. This  is more technical but appropriated to the case of Asian options. We may also generalize our results to  American options as it is done for frictionless models with Snell envelops.

The paper is organized as follows. The financial market is defined by a cost process, which is not necessarily convex, as described in Section \ref{Model}. Then, the dynamic programming principle is established in Section \ref{DPPP}, see Theorem \ref{DPP}. The last Section \ref{CF} is devoted to the implementation of the  dynamic programming principle. Precisely, we formulate results that ensure the propagation of the lower semicontinuity property to the minimal hedging cost  at any time, e.g. with respect to the spot price, see Theorem \ref{Meas+lsc+conv}, Corollary \ref{CoroPropag1}, Theorem \ref{PropagTheo}, Theorem \ref{LSC+} and Theorem \ref{PropagTheo1}. In Subsection \ref{FC}, fixed cost models are considered.   Theorem \ref{PropagTheoHorizon} also states the propagation of the lower semicontinuity property that allows to numerically compute the minimal hedging cost backwardly. It is formulated under a no-arbitrage condition on the enlarged market only composed of linear transaction costs in the spirit of \cite{LT} but also \cite{Pennanen} in the context of utility maximization.

\section{Financial market model defined by a cost process}\label{Model}
We consider a stochastic basis in discrete-time $(\Omega, (\cF_t)_{t=0}^T,\bP)$ where the filtration $(\cF_t)_{t=0}^T$ is  complete, i.e.  $\cF_0$ contains the negligible sets for $\bP$. By convention, we also define $\cF_{-1}:=\cF_0$. If $A$ is a random subset of $\R^d$, $d\ge 1$, we denote by $L^0(A,\R^d)$ the family of (equivalence classes of) all random variables $X$ (defined up to a negligible set) such that $X(\omega)\in A(\omega)$, $\bP$ a.s.$(\omega)$. It is well known that, if $A(\omega)\ne \emptyset$ $P$ a.s.$(\omega)$ and if $A$ is graph-measurable, see \cite{Mol}, then $L^0(A,\R^d)\ne \emptyset$. When using this property, we refer it  saying  {\sl by measurable selection arguments}, as it is usual to do when claiming the existence of $X\in L^0(\R^d,\cF)$ such that $X\in A$ a.s..

We also adopt the following notations. We denote by ${\rm int }A$ the interior of any $A\subseteq \R^d$ and ${\rm cl}A$ is its closure. The positive dual of $A$ is defined as $A^*:=\{x\in \R^d:~ax\ge 0,\,\forall a\in A\}$ where $ax$ designates the Euclidean scalar product of $\R^d$. At last, if $r\ge 0$, we denote by 
$\bar B(0,r)\subseteq \R^d$ the closed ball of all $x\in \R^d$ such that the norm  satisfies  $|x|\le r$.\smallskip

We consider a financial market where transaction costs are charged when the agents buy or sell risky assets. The typical case is a model defined by a bond whose discounted price is $S^1 =1$ and $d-1$ risky assets that may be traded at some bid and ask discounted prices $S^b$ and $S^a$, respectively, when selling or buying.  We refer the readers to the huge literature on models with transactions costs, in particular see \cite{KS}.\smallskip

Our general model is defined by a set-valued process $(\G_t)_{t=0}^T$ adapted to the filtration  $(\cF_t)_{t=0}^T$. Precisely, we suppose that for all $t\le T$, $\G_t$ is $\cF_t$-measurable in the sense of the graph 
${\rm Graph}(\G_t)=\{(\omega,x):~x\in \G_t(\omega)\}$ that belongs to $\cF_t\otimes \cB(\R^d)$,
where $\cB(\R^d)$ is the Borel $\sigma$-algebra on $\R^d$ and $d\ge 1$ is the number of  assets. 

We suppose that $\G_t(\omega)$ is closed for  every $\omega\in \Omega$ and $\G_t(\omega)+\R^d_+\subseteq \G_t(\omega)$, for all $t\le T$. The cost value process $\C=(\C_t)_{t=0}^T$ associated to $\G$ is defined as:
$$\C_t(z)=\inf\{\alpha \in \R:~ \alpha e_1-z\in \G_t\}=\min\{\alpha \in \R:~ \alpha e_1-z\in \G_t\},\quad z\in \R^d.$$
We suppose that the right hand side in the definition above is non empty a.s. and $-e_1$ does not belong to $\G_t$ a.s. where $e_1=(1,0,\cdots,0)\in \R^d$.  Moreover, by assumption, $\C_t(z)e_1-z\in \G_t$ a.s. for all $z\in \R^d$. Note that $\C_t(z)$ is the minimal amount of cash one needs to buy the financial position $z\in \R^d$ at time $t$. In particular, we suppose that $\C_t(0)=0$. \smallskip

Similarly, we may define the liquidation value process $\L=(\L_t)_{t=0}^T$ associated to $\G$ as:
\bean
    \L_t(z) := \sup\left\{\alpha\in\R:z-\alpha e_1\in \G_t\right\},\quad z\in \R^d.
\eean
We observe that $\L_t(z)=-\C_t(-z)$ and $\G_t=\{z\in \R^d:~\L_t(z)\ge 0\}$ so that our model is equivalently defined by $\L$ or $\C$. Note that $\G_t$ is closed if and only if $\L_t(z)$ is upper semicontinuous (u.s.c.) in $z$, see \cite{LT}, or equivalently $\C_t(z)$ is lower semicontinuous (l.s.c.) in $z$. Naturally,  $\C_t(z)=\C_t(S_t,z)$ depends on the available quantities and prices for the risky assets, described by an exogenous vector-valued $\cF_t$-measurable random variable $S_t$ of $\R^m_+$, $m\ge d$, and also depends on the quantities $z\in \R^d$ to be traded. Here, we suppose that $m\ge d$ as an asset may be described by 
several prices and quantities offered by the market, e.g. bid and ask prices, or several pair of bid and ask prices of an order book and the associated quantities offered by the market. \smallskip

In the following, we suppose the following assumptions on the cost process $\C$. For any $t\le T$, the cost function $\C_t$ is a lower semicontinuous Borel function defined on $\R^m\times \R^d$ such that
\begin{align*}
    &\C_t(s,0) = 0,\,\forall s\in \R^m_+,\\
    &\C_t(s,x+\lambda e_1) = \C_t(s,x) + \lambda,\,\lambda\in\R,\, x\in\R^d,\, s\in \R^m_+\,\ ( {\rm cash\, invariance}),\\ 
    & \C_T(s,x_2)\ge \C_T(s,x_1),\, \forall x_1,x_2\, {\rm\, s.t. }\, x_2-x_1\in \R^d_+ \ ( \C_T\, {\rm is \, increasing\, w.r.t.}\, \R^d_+),\\
    &|\C_t(s,x)| \le h_t(s,x),
\end{align*}
where $h_t$ is a deterministic continuous function.  Note that  $\C_T$  is increasing w.r.t. $\R^d_+$ is equivalent to $\G_T+\R^d_+\subseteq \G_T$. Moreover, if $\delta$ is an increasing bijection from $[0,+\infty]$ to $[0,+\infty]$ such that $\delta(0)=0$ and $\delta(\infty)=\infty$, we say that $\C_t$ is positively super $\delta$-homogeneous if the following property holds:

\bean \C_t(s,\lambda x) \geq \delta(\lambda) C_t(s,x),\forall \lambda\geq 1,\, s\in \R^m_+,\, x\in\R^d.\eean

A classical case is when $\delta(x)=x$ and the positive homogeneous property holds, e.g. for models with proportional transaction costs, as the solvency set process $\G$ is a positive cone, see \cite{KS}. More generally, if $\C_t(s,x)$ is convex in $x$ and $\C_t(s,0) = 0$, it is clear that $\C_t$ is positively super $\delta$-homogeneous with $\delta(x)=x$. Actually, in our definition, the domain of validity $\lambda\ge 1$ may be replaced by $\lambda\ge r$ where $r>0$ is arbitrarily chosen. In that case, all the results we formulate in this paper are still valid.  We now present a typical model that satisfies our assumptions:

\begin{ex}[Order book]\label{OB}{\rm Suppose that the financial market is defined by an order book. In that case, we define $S_t$, at any time $t$,  as
$$S_t=((S^{b,i,j}_t,S^{a,i,j}_t),(N^{b,i,j}_t,N^{a,i,j}_t))_{i=1,\cdots,d,j=1,\cdots,k},$$
where $k$ is the order book's depth and, for each $i=1,\cdots,d$, $S^{b,i,j}_t,S^{a,i,j}_t$ are the bid and ask prices for asset $i$ in the $j$-th line of the order book and $(N^{b,i,j}_t,N^{a,i,j}_t)\in (0,\infty)^2$ are the available quantities for these  bid and ask prices. We suppose that $N^{b,i,k}_t=N^{a,i,k}_t=+\infty$ so that the market is completely liquid.  By definition of the order book, we have $S^{b,i,1}_t>S^{b,i,2}_t>\cdots>S^{b,i,k}_t$  and $S^{a,i,1}_t<S^{a,i,2}_t<\cdots<S^{a,i,k}_t$.  We then define the cost function as 
$$\C_t(x)=x^1+\sum_{i=2}^d\C_t^i(x^i),\quad x=(x^1,\cdots,x^d)\in \R^d.$$ With the convention $\sum_{r=1}^j=0$ if $j=0$, we consider the cumulated quantities  $Q_t^{a,i,j}:=\sum_{r=1}^jN^{a,i,r}_t$, $j=0,\cdots,k,$ the same for  $Q_t^{b,i,j}$. We have:
\bean \C_t^i(y)&=&\sum_{r=1}^jN^{a,i,r}_tS^{a,i,r}_t+(y-Q_t^{a,i,j})S^{a,i,j+1}_t,\quad {\rm if\,} Q_t^{a,i,j}<y\le Q_t^{a,i,j+1},\\
\C_t^i(y)&=&-\sum_{r=1}^jN^{b,i,r}_tS^{b,i,r}_t+(y+Q_t^{b,i,j})S^{b,i,j+1}_t,\quad {\rm if\,} -Q_t^{b,i,j+1}<y\le -Q_t^{b,i,j}.
\eean
Note that the first expression of $\C_t^i(z)$ above corresponds to the case where we buy $y>0$ units of asset $i$. The second expression  is $\C_t^i(y)=-\L_t^i(-y)$ when $y<0$ so that $-\C^i_t(y)$ is the liquidation value of the position $-y$, i.e. by selling the quantity $-y>0$ at the bid prices. We observe that $\C_t^i(y)$ is a convex function in $y$ satisfying the cash invariance, such that $\C^i_t(0)=0$ and, at last, we show that $\C_t^i$ is positively super homogeneous as defined above. 

To do so, we first consider $y>0$ and we show that $\C^i_t(\lambda y)\ge \lambda \C^i_t(y)$ for $\lambda>1$ by induction on the interval $]Q_t^{a,i,j},Q_t^{a,i,j+1}]$ that contains $y$.   For $j=1$, $\C^i_t(y)=S^{a,i,1}_ty$ and $\C_t^i(\lambda y)= \C_t^i(Q_t^{a,i,j_{\lambda}})+(\lambda y-  Q_t^{a,i,j_{\lambda}})S^{a,i,j_{\lambda}+1}_t$ where $j_{\lambda}$ is such that $\lambda y\in ]Q_t^{a,i,j_{\lambda}},Q_t^{a,i,j_{\lambda}+1}]$. As $S^{a,i,1}_t$ is the smallest ask price, we get that $\C^i_t(Q_t^{a,i,j_{\lambda}})\ge Q_t^{a,i,j_{\lambda}}S^{a,i,1}_t$ and $(y-  Q_t^{a,i,j_{\lambda}})S^{a,i,j_{\lambda}+1}_t\ge (\lambda y-  Q_t^{a,i,j_{\lambda}})S^{a,i,1}_t$. We deduce that $\C^i_t(\lambda y)\ge \lambda yS^{a,i,1}_t$ hence $\C^i_t(\lambda y)\ge \lambda \C^i_t(y)$. More generally, if $y\in ]Q_t^{a,i,j},Q_t^{a,i,j+1}]$, $\lambda y> \lambda Q_t^{a,i,j}$ hence $\C_t(\lambda y)\ge \C_t(\lambda Q_t^{a,i,j})+(\lambda y-\lambda Q_t^{a,i,j})S_t^{a,i,\tilde j}$ where $\tilde j$ is such that $ Q_t^{a,i,\tilde j}<\lambda Q_t^{a,i,j}\le Q_t^{a,i,\tilde j+1}$. Indeed, the extra  quantity $\lambda y-\lambda Q_t^{a,i,j}$ is bought at a price larger than or equal to the maximal ask price $S_t^{a,i,\tilde j}$  when buying the quantity $\lambda Q_t^{a,i,j}$.     As $\lambda Q_t^{a,i,j}>Q_t^{a,i,j}$, we deduce that $\tilde j\ge j+1$. Using the induction hypothesis, we have $\C^i_t(\lambda Q_t^{a,i,j})\ge \lambda \C^i_t( Q_t^{a,i,j})$ and we deduce that 
\bean
\C^i_t(\lambda y)\ge \lambda C_t^i( Q_t^{a,i,j})+(\lambda y-\lambda Q_t^{a,i,j})S_t^{a,i,j+1}=\lambda \C^i_t(y).
\eean
By the same reasoning, $\L_t^i(\lambda y)\le \lambda \L_t^i(y)$ if $y>0$ with  $\L_t^i(y)=-\C_t^i(-y)$. Therefore, we also get that $\C_t^i(\lambda y)\ge \lambda \C^i_t(y)$ for $\lambda>1$ and $y<0$. \smallskip

We finally conclude that the cost process $\C$ satisfies the conditions we impose above. In particular, notice that $\C_t(s,z)$ is continuous in $(s,z)$. $\triangle$

}
\end{ex}

A portfolio process is by definition a stochastic process $(V_t)_{t=-1}^T$ where $V_{-1}\in \R e_1$ is the initial endowment expressed in cash that we may convert immediately into $V_0\in \R^d$ at time $t=0$. By definition, we suppose that 
$$\Delta V_t=V_t-V_{t-1}\in -\G_t,\, a.s.,\quad t=0,\cdots,T.$$

This means that any position $V_{t-1}=V_t +(-\Delta V_t)$ may be changed into the new position $V_t$, letting aside the residual part $(-\Delta V_t)$ that can be liquidated without any debt, i.e. $\L_t(-\Delta V_t)\ge 0$. Notice that, super-hedging or hedging a terminal claim is mainly equivalent in our setting as it is allowed to throw money, i.e. we may have $\L_t(-\Delta V_t)> 0$. \smallskip

\section{Dynamic programming principle for pricing}\label{DPPP}

Let $\xi\in L^0(\R^d,\mathcal{F}_T)$ be a contingent claim. Our goal is to characterize the set of all portfolio processes $(V_t)_{t=-1}^T$ such that $V_T= \xi$, as defined in the last section. We are mainly interested by the infimum cost one needs to hedge $\xi$, i.e. the infimum value of the initial capitals $V_{-1}e_1\in \R$ among the portfolios $(V_t)_{t=-1}^T$ replicating $\xi$. \smallskip

In the following, we use the notation $z = (z^1,z^2,...,z^d)\in \R^d$ and we denote  $z^{(2)} = (z^2,...,z^d)$. We shall heavily use the notion of $\mathcal{F}_{t}$-measurable conditional essential supremum (resp. infimum) of a family of random variables, i.e.  the smallest (resp. largest) $\mathcal{F}_{t}$-measurable random variable that dominates (resp. is dominated by) the family with respect to the natural order between $[-\infty,\infty]$-valued random variables, i.e. $X\le Y$ if $\bP(X\le Y)=1$, see \cite[Section 5.3.1]{KS}. 

\subsection{The one step hedging problem}

 Recall that $V_{T-1}\ge_{\G_T}V_T$ by  definition of a portfolio process. Then, the  hedging problem $V_T=\xi$ \footnote{The problem $V_T\ge_{\G_T} \xi$  is equivalent to our one  if  $\G_T+\G_T\subseteq \G_T$. In general, any $V_T$ such that $V_T\ge_{\G_T} \xi$ may be changed into $\xi$ through an additional cost. So, the formulation $V_T=\xi$ is chosen as we are interested in minimal costs. } is  equivalent at time $T-1$ to:
\bean
    \L_T(V_{T-1}) \geq \xi &&\Longleftrightarrow V^1_{T-1} \geq \xi^1 - \L_T((0,V^{(2)}_{T-1})),\\
    &&\Longleftrightarrow V^1_{T-1} \geq \esssup_{\mathcal{F}_{T-1}}\left(\xi^1 - \L_T((0,V^{(2)}_{T-1} - \xi^{(2)}))\right),\\
    &&\Longleftrightarrow V^1_{T-1} \geq \esssup_{\mathcal{F}_{T-1}}\left(\xi^1 + \C_T((0,\xi^{(2)}- V^{(2)}_{T-1}))\right),\\
    &&\Longleftrightarrow V^1_{T-1} \geq F_{T-1}^{\xi}(V^{(2)}_{T-1}),
\eean
\smallskip
where 
\bea\label{Ftxi}
    F_{T-1}^{\xi}(y) &:=& \esssup_{\mathcal{F}_{T-1}}\left(\xi^1 + \C_T((0,\xi^{(2)}- y))\right).
\eea

By virtue of Proposition \ref{JointlyMeasVersion} in Appendix,  we may suppose that  $F_{T-1}^{\xi}(\omega,y)$ is jointly $\cF_{T-1}\times \cB(\R^{d-1})$-measurable,  l.s.c. as a function of $y$ and convex if $\C_T(s,y)$ is convex in $y$. As $\mathcal{F}_{T-1}$ is supposed to be complete, we conclude that  $F_{T-1}^{\xi}$  is an $\mathcal{F}_{T-1}$ normal integrand, see Definition \ref{NI}  and \cite{RW}. \smallskip

\subsection{The multi-step hedging problem}

 We denote by $\cP_t(\xi)$ the set of all portfolio processes starting at time $t\le T$  that replicates $\xi$ at the terminal date $T$:
\begin{align*}
    \cR_t(\xi):=\left\{(V_s)_{s=t}^T,-\Delta V_{s}\in L^0(\G_{s},\cF_s),\, \forall s\geq t+1, V_T=\xi\right\}.
\end{align*}
The set of replicating prices of $\xi$ at time $t$ is
\begin{align*}
    \cP_{t}(\xi):= \left\{V_t = (V_t^1,V_t^{(2)}): (V_s)_{s=t}^T\in\cR_{t}(\xi)\right\}.
\end{align*}
The infimum replicating cost is then defined as:
\begin{align*}
    c_{t}(\xi): = \essinf_{\cF_t}\left\{\C_t(V_t),\, V_t\in\cP_{t}(\xi)\right\}.
\end{align*}
By the previous section, we know that  $V_{T-1}\in \cP_{T-1}(\xi)$ if and only if
\begin{align*}
    V_{T-1}^1\geq\esssup_{\cF_{T-1}}\left(\xi^1 + \C_T(0,\xi^{(2)}-V_{T-1}^{(2)})\right) \text{ a.s..}
\end{align*}
Similarly, $V_{T-2}\in \cR_{T-2}(\xi)$ if and only if there exists $V_{T-1}^{(2)}\in L^0(\R^{d-1},\cF_{T-1})$ such that
\begin{align*}
    V_{T-2}^1 \geq\esssup_{\cF_{T-2}}\left(\esssup_{\cF_{T-1}}\left(\xi^1 + \C_T(0,\xi^{(2)}-V_{T-1}^{(2)})\right) + \C_{T-1}(0,V_{T-1}^{(2)} - V_{T-2}^{(2)})\right).
\end{align*}
As the conditional essential supremum operator satisfies the tower property, we deduce  that $V_{T-2}\in \cR_{T-2}(\xi)$ if and only if there is $V_{T-1}^{(2)}\in L^0(\R^{d-1},\cF_{T-1})$ such that
\begin{align*}
    V_{T-2}^1 \geq\esssup_{\cF_{T-2}}\left(\xi^1 + \C_T(0,\xi^{(2)}-V_{T-1}^{(2)}) + \C_{T-1}(0,V_{T-1}^{(2)} - V_{T-2}^{(2)})\right).
\end{align*}
Recursively, we get that $V_t\in \cP_{t}(\xi)$ if and only if, for some $V_{s}^{(2)}\in L^0(\R^{d-1},\cF_{s})$, $s=t+1,\cdots,T-1$, and $V_T^{(2)}=\xi^{(2)}$, we have
\begin{align*}
    V_t^1\geq \esssup_{\cF_{t}}\left(\xi^1 +\sum_{s=t+1}^{T}\C_{s}(0,V_{s}^{(2)} - V_{s-1}^{(2)})\right).
\end{align*}
In the following, for $u\le T-1$, $\xi_{u-1}\in L^0(\R^d,\cF_{u-1})$, and $\xi\in L^0(\R^d,\cF_T)$, we introduce the sets  $$\Pi_{u}^T(\xi_{u-1},\xi):=\{\xi_{u-1}^{(2)}\}\times \Pi_{s=u}^{T-1}L^0(\R^{d-1},\cF_{s})\times \{\xi^{(2)}\}$$  of all families $(V^{(2)}_s)_{s=u-1}^{t+1}$ such  that $V^{(2)}_{u-1}=\xi_{u-1}^{(2)}$, $V^{(2)}_s\in L^0(\R^{d-1},\cF_{s}) $ for all $s=u,\cdots,T-1$ and $V^{(2)}_{T}=\xi^{(2)}$. We set $\Pi_u^T(\xi):=\Pi_u^T(0,\xi)=\Pi_u^T(\xi_{u-1},\xi)$ when $\xi_{u-1}^{(2)}=0$. When $u=T$, we set   $\Pi_T^T(\xi_{T-1},\xi):=\{\xi_{T-1}^{(2)}\}\times  \{\xi^{(2)}\}$. Therefore, the infimum replicating cost at time $0$ is given by
\bean
c_{0}(\xi) = \mathop{\essinf_{\cF_{0}}}\limits_{V^2\in\Pi_0^T(\xi)}  \mathop{\esssup_{\cF_{0}}}\bigg(\xi^1 +\sum_{s=0}^{T}\C_{s}(0,V_{s}^2 - V_{s-1}^2)\bigg).
\eean

For $0\le t\le T$ and $V_{t-1}\in L^0(\R^d,\cF_{t})$, we define $\gamma_{t}^{\xi}(V_{t-1})$ as:
\begin{align*}
    \gamma_{t}^{\xi} (V_{t-1}):= \mathop{\essinf_{\cF_{t}}}\limits_{V^{(2)}\in \Pi_t^T(V_{t-1},\xi)} \esssup_{\cF_{t}}\bigg(\xi^1 
    +\sum_{s=t}^{T}\C_{s}(0,V_{s}^{(2)} - V_{s-1}^{(2)})\bigg).
\end{align*}

Note that $\gamma_{t}^{\xi} (V_{t-1})$ is the infimum cost to replicate the payoff $\xi$ when starting from the initial risky position $(0,V_{t-1}^{(2)})$ at time $t$. Observe that $\gamma_{t}^{\xi} (V_{t-1})$ does not depend on the first component $V_{t-1}^1$. Moreover, $$\gamma^{\xi} _{T}(V_{T-1}) = \xi^1+\C_T(0,\xi^{(2)}-V_{T-1}^{(2)}).$$ 
As $\G_T+\R^d_+\subseteq \G_T$, we also observe that $\gamma^{\xi} _{T}(V_{T-1})\ge \gamma^{0} _{T}(V_{T-1})$. At last, observe that $c_{0}(\xi)=\gamma_0^{\xi} (0)$. Therefore, the main goal of our paper is to study the random functions $(\gamma^{\xi} _{t})_{t=0,1,\cdots, T}$ and to propose conditions under which it is possible to compute them backwardly so that we may estimate $c_{0}(\xi)$. The main contribution of this section is the following:

\begin{theo}[Dynamic Programming Principle]\label{DPP}
For any $0\leq t\leq T-1$ and $V_{t-1}\in L^0(\R^d,\cF_{t-1})$, we have
\begin{align} \label{DPP1}
    \gamma_{t}^{\xi} (V_{t-1}) = \mathop{\essinf_{\cF_{t}}}\limits_{V_{t}\in L^0(\R^{d},\cF_{t})}\esssup_{\cF_{t}}\left(\C_{t}(0,V_{t}^{(2)}- V_{t-1}^{(2)}) + \gamma_{t+1}^{\xi} (V_{t})\right).
\end{align}\end{theo}
\begin{proof}
We denote the right hand side of (\ref{DPP1}) by $\bar{\gamma}^{\xi} _{t}(V_{t-1})$. We first verify  (\ref{DPP1})  for $t=T-1$. Recall  that $\gamma^{\xi} _{T}(V_{T-1}) = \xi^1+\C_T(0,\xi^{(2)}-V_{T-1}^{(2)})$ if $V_{T-1}$ belongs to $L^0(\R^{d},\cF_{T-1})$. It is clear that  (\ref{DPP1}) holds  for $t=T-1$ by definition of $\gamma_{T-1}^{\xi} (V_{T-1})$. By induction, let us show that (\ref{DPP1}) holds at time $t$   if this holds at time $t+1$.  Let us define
$$f_t(V_{t-1},V_{t}):=\esssup_{\cF_{t}}\left(\C_{t}(0,V_{t}^{(2)}- V_{t-1}^{(2)}) + \gamma_{t+1}^{\xi} (V_{t})\right), t\le T-1.$$
  We observe that the collection of random variables 
  $$\Gamma_{t}=\{f_t(V_{t-1},V_{t}):~V_{t}\in L^0(\R^{d},\cF_{t})\}$$ is directed downward, i.e. if $f_t^j=f_t(V_{t-1},V_{t}^j)\in \Gamma_{t}$, $j=1,2$, then there exists $f_t\in \Gamma_{t}$ such that $f_t\le f_t^1\wedge f_t^2$. Indeed, to see it, it suffices to consider $f_t=f_t(V_{t-1},V_{t})$ where $V_{t}=V_{t}^11_{\{f_t^1\le f_t^2\}}+V_{t}^21_{\{f_t^1>f_t^2\}}$. Therefore, there exists a sequence $(V_{t}^n)_{n\ge 1}\in L^0(\R^{d},\cF_{t})$ such that $\bar{\gamma}^{\xi} _{t}(V_{t-1})=\inf_n f_t(V_{t-1},V_{t}^n)$, see \cite[Section 5.3.1]{KS}. We deduce for any $\epsilon>0$, the existence of $\tilde V_{t}\in L^0(\R^d,\cF_t)$ such that
$\bar{\gamma}_{t}^{\xi} (V_{t-1})+\epsilon\ge f_t(V_{t-1}^{(2)},\tilde V_{t}^{(2)})$. Similarly, by forward iteration, using the induction hypothesis $\gamma_{r}^{\xi} (\tilde V_{r-1})=\bar{\gamma}_{r}^{\xi} (\tilde V_{r-1})$, $r\ge t+1$, we obtain  the existence of $\tilde V_{r}\in L^0(\R^d,\cF_{r})$ such that $\gamma_{r}^{\xi} (\tilde V_{r-1})+\epsilon\ge f_{r}(\tilde V_{r-1}^{(2)},\tilde V_{r}^{(2)})$, for all $r=t+1,\cdots, T-1$.
With $\tilde V_{t-1}=V_{t-1}$ and $\tilde V_{T}=\xi$, we deduce that

$$\bar{\gamma}^{\xi} _{t}(V_{t-1})+\epsilon T\ge \esssup_{F_t}\left ( \xi^1+\sum_{s=t}^{T}\C_{s}(0,\tilde V_{s}^{(2)} - \tilde V_{s-1}^{(2)})\right)\ge \gamma^{\xi} _t(V_{t-1}).$$

As $\epsilon$ goes to $0$, we conclude that $\bar{\gamma}_{t}^{\xi} (V_{t-1})\ge \gamma_t^{\xi} (V_{t-1})$ . The reverse inequality is easily obtained by induction and using the assumption that $\bar{\gamma}^{\xi} _{r}$ and $\gamma^{\xi} _t$ coincide if $r\ge t$ with the tower property. The conclusion follows. \end{proof}

\section{Computational feasibility of the dynamic programming principle}\label{CF}

The dynamic programming principle (\ref{DPP1}) allows to get $\gamma_{t}^{\xi} (V_{t-1})$ from the cost function $\C_t$ and from $\gamma_{t+1}^{\xi} $.  In this section, our first main contribution is to show that $\gamma_{t}^{\xi}$ is l.s.c. for any $t$ and convex if  the cost functions are. Then, we formulate some results  allowing to compute $\omega$-wise the essential supremum and the essential  infimum of (\ref{DPP1}). \smallskip

As the term $\C_{t}(0,V_{t}^{(2)}- V_{t-1}^{(2)})$ in (\ref{DPP1}) is $\cF_{t}$-measurable,   it is sufficient to consider the conditional supremum 

$$\theta_{t}^\xi(V_{t}):= \esssup_{\cF_{t}}\gamma_{t+1}^\xi(V_{t})$$
to compute the essential supremum of (\ref{DPP1}). In the following, we shall use the following notations:
\bea
D_t^\xi(V_{t-1},V_{t}):=\C_{t}((0,V_{t}^{(2)}- V_{t-1}^{(2)})) + \theta_{t}^\xi(V_{t}),\label{D1}\\
D_t^\xi(S_t,V_{t-1},V_{t}):=\C_{t}(S_t,(0,V_{t}^{(2)}- V_{t-1}^{(2)})) + \theta_{t}^\xi(S_t,V_{t})\label{D2}.
\eea
The second notation is used when we stress the dependence on $S_t$.

\subsection{Computational feasibility for convex costs}\label{CF-convex}

The following first result ensures the propagation of the lower semicontinuity and convexity of  the random function $\gamma_{t+1}^{\xi}$ to  $\gamma_{t}^{\xi}$ as we shall see in Theorem \ref{Meas+lsc+conv}. This is a crucial property to  pointwisely compute the essential infimum in (\ref{DPP1}).

\begin{prop}\label{Meas+lsc} Suppose that there exists a random $\cF_{t+1}$-measurable lower semicontinuous  function $\tilde \gamma_{t+1}^{\xi}$ defined on $\R^d$ such that   $\gamma_{t+1}^{\xi}(V_{t})=\tilde \gamma_{t+1}^{\xi}(V_t)$ for all $V_t\in L^0(\R^d,\cF_t)$. Then,  there exists  a random $\cF_{t}$-measurable lower semicontinuous function $ \tilde\theta_{t}^\xi$ defined on $\R^d$ such that  $\theta_{t}^\xi(V_{t})=\tilde\theta_{t}^\xi(V_t)$ for all  $V_t\in L^0(\R^d,\cF_t)$. Moreover, the random function $y\mapsto \tilde\theta_{t}^\xi(y)$ is a.s. convex if $y\mapsto \tilde \gamma_{t+1}^{\xi}(y)$ is a.s. convex.
\end{prop}
\begin{proof}
We consider the random function 
$$f(z)=z^1+\tilde \gamma_{t+1}^{\xi}((0,z^{(2)}))=z^1+f((0,z^{(2)})),\quad z\in \R^d.$$ We have $\gamma_{t+1}^{\xi}(V_{t})=f((0,V_t^{(2)}))$ so it suffices to apply Proposition  \ref{JointlyMeasVersion}.
\end{proof}

In order to numerically compute the minimal costs, we need to impose the finiteness of $\gamma_{t}^{\xi}(V_{t-1})$, i.e. $\gamma_{t}^{\xi}(V_{t-1})>-\infty$,  at any time $t$,  and for all $V_{t-1}\in L^0(\R^d,\cF_{t-1})$. This is why we introduce the following condition:
\begin{defi} We say that the financial market satisfies the Absence of Early Profit condition (AEP) if, at any time $t\le T$, and for all $V_{t}\in L^0(\R^d,\cF_{t})$, $\gamma_{t}^{0}(V_{t})>-\infty$ a.s..
\end{defi}

\begin{rem}\quad\smallskip{\rm

\noindent 1.) Let us comment the condition AEP. Suppose that AEP does not hold, i.e. there is $V_{t}\in L^0(\R^d,\cF_{t})$ such that   $\Lambda_t=\{\gamma_{t}^{0}(V_{t})=-\infty\}$ satisfies $P(\Lambda_t)>0$. Any arbitrarily chosen amount of cash $-n<0$ allows to hedge the zero payoff at time $t$ on $\Lambda_t$ when starting from the initial position $(0,V_{t}^{2})$ by definition of  $\gamma_{t}^{0}(V_{t})=-\infty$. Then, at time $t$, we may obtain an arbitrarily large profit on $\Lambda_t$ as follows: We write $0=\left((0,V_{t}^{2})-ne_1\right)1_{\Lambda_t} + a_{t-1}^n$ where  $a_{t-1}^n=\left(ne_1-(0,V_{t}^{2})\right)1_{\Lambda_t}$.  The position $(0,V_{t}^{2})-ne_1$ allows to get the zero claim at time $T$. Moreover,  $\L_t(a_{t-1}^n)=n1_{\Lambda_t}+\L_t((0,V_{t}^{2}))1_{\Lambda_t}$ tends to $+\infty$ as $n\to \infty$ on $\Lambda_t$, i.e. it is possible to make an early profit at time $t$, as large as possible. 

\noindent 2.) If $\xi\in L^0(\R^d_+,\cF_T)$, then  $\gamma^{\xi} _{t}(V_{t-1})\ge \gamma^{0} _{t}(V_{t-1})>-\infty$ under AEP.

\noindent 3.) Under Assumptions \ref{g} and  \ref{CondSuppHemmiCont}  below, condition AEP holds by Lemma \ref{ContinuousLowerBoundGamma}.} $\triangle$

\end{rem}

\begin{Assum}\label{assumption payoff} The payoff $\xi$ is hedgeable, i.e. there exists a portfolio process $(V_u^{\xi})_{u=0}^T$ such that $\xi=V_T^{\xi}$. 
\end{Assum}

\begin{lemm} \label{LemmaGammaFini} Under Assumption \ref{assumption payoff}, $\gamma^{\xi} _{t}(V_{t-1})<\infty$ for all $V_{t-1}\in L^0(\R^d,\cF_{t})$.
\end{lemm}
\begin{proof} We observe that the amount of capital $\alpha_t=\C_t(V_t^{\xi}-(0,V_{t-1}^{(2)}))$ allows one to get the position $V_t^{\xi}-(0,V_{t-1}^{(2)})$. Therefore, starting from the initial position $(0,V_{t-1}^{(2)})$, the capital $\C_t(V_t^{\xi}-(0,V_{t-1}^{(2)}))$ is enough to get $V_t^{\xi}$ and then   $\xi$ at time $T$ since $V_T^{\xi}=\xi$. We then deduce that 
$$\gamma^{\xi} _{t}(V_{t-1})\le \alpha_t\le h_t(S_t,V_t^{\xi}-(0,V_{t-1}^{(2)}))<\infty.$$
\end{proof}
The following theorem states that the convexity and lower semicontinuity properties propagate backwardly from $\gamma_{t+1}^{\xi}$ to $\gamma_{t}^{\xi}$.

\begin{theo}\label{Meas+lsc+conv} Suppose that Assumption \ref{assumption payoff} and Condition AEP hold. Suppose that there exists a random $\cF_{t+1}$-normal convex integrand $\tilde \gamma_{t+1}^{\xi}$ defined on $\R^d$ such that   $\gamma_{t+1}^{\xi}(V_{t})=\tilde \gamma_{t+1}^{\xi}(V_t)$ for all $V_t\in L^0(\R^d,\cF_t)$. Suppose that the cost function $\C_t(s,z)$ is convex in $z$. Then,  there exists  a random $\cF_{t}$- normal convex integrand $ \tilde\gamma_{t}^{\xi}$ defined on $\R^d$ such that  $\gamma_{t}^{\xi}(V_{t-1})=\tilde\gamma_{t}^{\xi}(V_{t-1})$ for all  $V_{t-1}\in L^0(\R^d,\cF_t)$ and we have:
$$ \tilde\gamma_{t}^{\xi}(v_{t-1})=\inf_{y\in \R^{d}}\left(\C_t(0,y^{(2)}-v_{t-1}^{(2)})+\tilde\theta_{t}^\xi(y) \right),$$
where $\tilde\theta_{t}^\xi$ is given by Proposition \ref{Meas+lsc}. In particular, $\tilde\gamma_t^\xi(\omega,x)\in\R$, for all $x\in \R^d$, a.s., so that  $\tilde\gamma_t^\xi(\omega,\cdot)$ is a continuous function a.s..
\end{theo}
\begin{proof} By Proposition \ref{Meas+lsc}, we deduce that $\theta_{t}^\xi(V_{t})=\tilde\theta_{t}^\xi(V_t)$ a.s. for every $V_t\in L^0(\R^d,\cF_t)$ where $ \tilde\theta_{t}^\xi$ is an $\cF_{t}$-normal convex integrand. Therefore, $\bar D_t(v_{t-1},v_t):=\C_t(0,v_t^{(2)}-v_{t-1}^{(2)})+\tilde\theta_{t}^\xi(v_t)$ is an $\cF_t$-normal integrand, convex in $(v_{t-1},v_t)$. By Lemma \ref{ess omega-wise lemm}, we have  $\tilde\gamma_t^\xi(V_{t-1})=\gamma_{t}^{\xi}(V_{t-1})$ a.s. for any $V_{t-1}\in L^0(\R^d,\cF_t)$.\smallskip

We claim that the mapping $(\omega,v_{t-1})\mapsto \tilde\gamma_t^\xi(v_{t-1})$ is $\cF_t\otimes\cB(\R^d)$-measurable. Indeed, since $\tilde D_t$ is convex and admits finite values in $\R$, we necessarily  have $\inf_{v_t\in\R^d}\tilde D_t(v_{t-1},v_t) = \inf_{v_t\in\Q^d}\tilde D_t(v_{t-1},v_t)$, and the measurability  follows. Next, we  show that $\tilde\gamma_t^\xi(\omega,\cdot)\in\R$ a.s.. First,  $\tilde\gamma_t^\xi(\omega,x)>-\infty$ for all $x\in \R^d$ a.s.. Otherwise, by a measurable selection argument, we may find an $\cF_t$-measurable selection $V_{t-1}$ such that $-\infty=\tilde\gamma_t^\xi(V_{t-1})=\gamma_{t}^{\xi}(V_{t-1})$ on a non null set. This is in contradiction with the AEP condition. Similarly, by Lemma \ref{LemmaGammaFini}, we  deduce that  $\tilde\gamma_t^\xi(\omega, x)<\infty$ for all $x\in \R^d$ a.s.. Therefore, the random function $\tilde\gamma_t^\xi(\omega,\cdot)$ only takes finite values a.s..

We finally conclude that the mapping $v_{t-1}\mapsto \tilde\gamma_{t}^{\xi}(v_{t-1})$ is a real-valued random convex function. In particular, $\tilde\gamma_{t}^{\xi}$ is continuous.
\end{proof}

\begin{rem}{\rm Suppose that the cost functions $\C_t(s,z)$, $t\le T$, are convex in $z$. Under Assumption \ref{assumption payoff}, as $\gamma^{\xi} _{T}(V_{T-1}) = \xi^1+\C_T(0,\xi^{(2)}-V_{T-1}^{(2)}))$ is l.s.c. and convex in $V_{T-1}$, we deduce that Theorem \ref{Meas+lsc+conv} applies backwardly step by step. In particular, it is possible to compute  $\gamma_{t}^{\xi}(v_{t-1})$ at any time $t$ as a $\omega$-wise infimum.} $\triangle$
\end{rem}

In the following, we consider  conditions under which it is possible to compute $\omega$-wisely the essential supremum $\theta_{t}^\xi$. The main ingredient is the knowledge of the conditional support ${\rm supp}_{\cF_t}S_{t+1}$ of $S_{t+1}$ knowing $\cF_t$. Recall that ${\rm supp}_{\cF_t}S_{t+1}$ is the smallest $\cF_t$-measurable random closed set that contains $S_{t+1}(\omega)$ a.s., see \cite{EL}.

\begin{Assum}\label{assumption support} 
For each $t\le T-1$, there exists a family of Borel functions $(\alpha^m_t)_{m\geq 1}$ defined on $\R^{m}$ such that ${\rm supp}_{\cF_t}S_{t+1}$ admits  the Castaing representation $(\alpha^m_t(S_t))_{m\geq 1}$,  i.e. ${\rm supp}_{\cF_t}S_{t+1}={\rm cl}(\alpha^m_t(S_t))_{m\geq 1}$. \end{Assum}

\begin{prop} \label{Propag1} Suppose that there exists a lower semicontinuous function $\tilde \gamma_{t+1}^{\xi}$ defined on $\R^m\times \R^d$ such that   $\gamma_{t+1}^{\xi}(V_{t})=\tilde \gamma_{t+1}^{\xi}(S_{t+1},V_t)$ for all $V_t\in L^0(\R^d,\cF_t)$. Then, $\theta_{t}^{\xi}(V_{t})=\sup_{z\in \text{supp}_{\cF_t}S_{t+1} }\tilde \gamma_{t+1}^{\xi}(z,V_{t})$. Moreover, under Assumption \ref{assumption support}, there exists a function $ \tilde\theta_{t}^\xi(s,v)$ defined on $(s,v)\in\R^m\times \R^d$, which is l.s.c. in $v$, such that  $\theta_{t}^\xi(V_{t})=\tilde\theta_{t}^\xi(S_t,V_t)$ for all  $V_t\in L^0(\R^d,\cF_t)$ and we have:
$$\tilde\theta_{t}^\xi(s,v):=\sup_m \tilde \gamma_{t+1}^{\xi}(\alpha_m(s),v)\quad (s,v)\in \R^m\times \R^d.$$
At last, $\tilde\theta_{t}^\xi(s,v)$ is l.s.c. in $(s,v)$ if the functions $(\alpha_m)_{m\ge 1}$ are continuous and, if $\tilde \gamma_{t+1}^{\xi}(s,v)$ is convex in $v$, then $\tilde\theta_{t}^\xi(s,v)$ is convex in $v$.
\end{prop}
\begin{proof} The proof is immediate by Proposition \ref{esssup} and  Lemma \ref{CountableRepGamma}. 
\end{proof}

\begin{Assum}\label{assumption support-1} 
For each $t\le T-1$, there exists a  family of   Borel functions  $(\alpha_{t}^m)_{m\ge 1}$ such that we have $S_{t+1}\in \{\alpha_{t}^m(S_{t}):~m\ge 1\}$ a.s. and such that $\bP(S_{t+1}=\alpha_{t}^m(S_{t})|\cF_{t})>0$ a.s. for all $m\ge 1$. 
\end{Assum}

\begin{prop} \label{Propag2} Suppose that there exists a Borel function $\tilde \gamma_{t+1}^{\xi}$ defined on $\R^m\times \R^d$ such that   $\gamma_{t+1}^{\xi}(V_{t})=\tilde \gamma_{t+1}^{\xi}(S_{t+1},V_t)$ for all $V_t\in L^0(\R^d,\cF_t)$. Then, under  Assumption \ref{assumption support-1} , there exists a Borel function $ \tilde\theta_{t}^\xi(s,v)$ defined on $(s,v)\in\R^m\times \R^d$ such that  $\theta_{t}^\xi(V_{t})=\tilde\theta_{t}^\xi(S_t,V_t)$ for all  $V_t\in L^0(\R^d,\cF_t)$ and we have:
$$\tilde\theta_{t}^\xi(s,v):=\sup_m \tilde \gamma_{t+1}^{\xi}(\alpha_m(s),v)\quad (s,v)\in \R^m\times \R^d.$$
\end{prop}
\begin{proof} The proof is immediate by Lemma \ref{CountSupp}. Note that we do not suppose that $\C_t$ is convex to obtain this result. \end{proof}

\begin{coro} \label{CoroPropag1}  Assume that the assumptions of Proposition \ref{Propag1} or Proposition \ref{Propag2}  hold    and Condition AEP holds. Suppose that  $\tilde \gamma_{t+1}^{\xi}(s,v)$ is convex in $v$. Then, $\gamma_{t}^{\xi}(V_{t-1})=\tilde \gamma_{t}^{\xi}(S_{t},V_{t-1})$ where $\tilde \gamma_{t}^{\xi}(s,v)$ is an $\cF_t$-normal integrand, convex in $v$. Moreover,

$$\tilde \gamma_{t}^{\xi}(s,v)=\inf_{y\in\R^{d}}\left(\C_t(s,(0,y^{(2)}-v^{(2)}))+\sup_m \tilde \gamma_{t+1}^{\xi}(\alpha_m(s),y) \right).$$

\end{coro}
\begin{proof} Under our assumptions, $\theta_{t}^\xi(V_{t})=\tilde\theta_{t}^\xi(S_t,V_t)$ for all  $V_t\in L^0(\R^d,\cF_t)$ where $ \tilde\theta_{t}^\xi(s,v)=\sup_m \tilde \gamma_{t+1}^{\xi}(\alpha_m(s),v) $ by Proposition \ref{Propag1} or Proposition \ref{Propag2}. As a supremum, $ \tilde\theta_{t}^\xi(s,v)$ is convex in $v$ if $\tilde \gamma_{t+1}^{\xi}(s,v)$ is. As $\C_t(s,y)$ is also  convex in $y$, we deduce that $D_t^{\xi}(y,v)=\C_t(s,(0,y^{(2)}-v^{(2)}))+\tilde\theta_{t}^\xi(s,y)$ is convex in $(y,v)$. Now, by arguing similarly to the proof of Theorem \ref{Meas+lsc+conv}, under AEP, $\tilde\gamma_t^\xi(v_{t-1})$ is a real- valued convex function in $v_{t-1}$  a.s..
\end{proof}

\subsection{Computational feasibility under strong AIP no-arbitrage condition}

The results of Section \ref{CF-convex} are not a priori sufficient to compute backwardly $\theta_{t-1}^\xi$ as we need $\gamma_{t}^\xi(s,v)$ to be l.s.c. in $s$, see Proposition \ref{Propag1}. This is why, we introduce the following conditions.

\begin{Assum}\label{g} The payoff function $\xi$ is of the form  $\xi = g(S_T)$, where $g\in \R^d_+ $ is continuous. Moreover, $\xi$ is hedgeable, i.e. there exists a portfolio process $(V_u^{\xi})_{u=0}^T$ such that $\xi=V_T^{\xi}$. 
\end{Assum}

\begin{Assum}\label{CondSuppHemmiCont} 
The conditional support is such that ${\rm supp}_{\cF_t}S_{t+1}=\phi_t(S_t)$ where $\phi_t$ is a set-valued lower hemicontinuous function, see Definition \ref{def: continuous}, with compact values such that $\phi_t(S_t)\subseteq \bar B(0,R_t(S_t))$ where $R_t$ is a continuous function on $\R^m$.
\end{Assum}

Note that under Assumption \ref{assumption support},  $\phi_t(S_t)={\rm cl}\{\alpha_m(S_t):~m\ge 1\}$  defines a   set-valued lower hemicontinuous function of $S_t$ if the functions $(\alpha_m)_{m\ge 1}$ are continuous, see Lemma \ref{lower hemicontinuous support}. In practice, we should be able to evaluate the conditional support from empirical data and deduce a continuous countable sense subset of it as a Castaing representation. Note that this representation is not unique.

\begin{defi}
We say that the  condition AIP holds at time $t$ if the minimal cost $c_t(0)=\gamma_t^0(0)$ of the  European zero claim $\xi=0$ is $0$ at time $t\le T$. We say that AIP holds if AIP holds at any time.
\end{defi} 

The condition AIP has been introduced for the first time in the paper \cite{BCL}. This is a weak no-arbitrage condition which is clearly satisfied in the real financial markets i.e. the price of a non negative payoff is non negative.

\begin{lemm} \label{LemmaAEP-AIP} Suppose that the cost functions are either sub-additive or super-additive. Then,  AIP implies AEP.
\end{lemm}
\begin{proof}
We prove it in the case where the cost function is sub-additive, the supper-additive case is similar. Suppose that AIP holds and $\C_t(s,v)$ is sub-additive in $v$.  For any $V_t,\tilde V_t\in L^0(\R^d,\cF_t)$, we have by the definition of $D_t^0$ (see \ref{D2}):
\begin{align*}
    D_t^0(S_t,V_t,\tilde V_t) &= \C_t(S_t,\tilde V_t - V_t) + \theta_t^0(S_t,\tilde V_t),\\
    &\geq \C_t(S_t,\tilde V_t) + \theta_t^0(S_t,\tilde V_t) - \C_t(S_t,V_t),\\
    &= D_t^0(S_t,0,\tilde V_t) - \C_t(S_t,V_t).
\end{align*}
Under AIP,  $D_t^0(S_t,0,\tilde V_t)\ge 0$ hence  $D_t^0(S_t,V_t,\tilde V_t)\geq -\C_t(S_t,V_t)$. We deduce that $\gamma_t^0(V_t) = \essinf_{\tilde V_t} D_t^0(S_t,V_t,\tilde V_t) \geq -\C_t(S_t,V_t) > -\infty$.
\end{proof}

\begin{defi}
We say that the  condition SAIP (Strong AIP condition) holds at time $t$ if AIP holds at time $t$  and, for any $Z_t\in L^0(\R^d,\cF_t)$, we have $D_t^0(S_t,0,Z_t)=0$ if and only if $Z_t^{(2)}=0$ a.s.. We say that SAIP holds if SAIP holds at any time.
\end{defi} 

Recall that $D_t^0(S_t,0,Z_t)$ is given by (\ref{D2}) and it is the  minimal cost expressed in cash that is needed at time $t$ to hedge the zero payoff when we start from the initial strategy $V_t=(\theta_t^0(Z_t),Z_t^{(2)})$, initial value of a portfolio process $(V_u)_{t\le u\le T}$ such that $V_T=0$. Therefore, the condition SAIP states that the minimal cost of the zero payoff is $0$ at time $t$ and this minimal cost is only attained  by the zero strategy $V_t=0$. This is intuitively clear as soon as any non null transaction implies positive costs. \smallskip

The following result is  our main contribution of this section: It states that the minimal cost function $\gamma_{t}^{\xi}$ is a l.s.c. function  of $S_t$ and $V_{t-1}$, i.e. $\gamma_{t}^{\xi}$ inherits from the lower semicontinuity of $\gamma_{t+1}^{\xi}$, under Assumptions \ref{g} and  \ref{CondSuppHemmiCont}, if SAIP holds as we shall see. We introduce the notation
$$S^{d-1}(0,1)=\{z\in \R^d:~z^1=0\,{\rm and\,}|z|=1\}.$$

\begin{theo} \label{PropagTheo} Suppose that $\C_t$ is positively super $\delta$-homogeneous.  Suppose that there exists a $\cF_{t+1}$-normal integrand $\tilde \gamma_{t+1}^{\xi}$ defined on $\Omega\times\R^m\times \R^d$ such that   $\gamma_{t+1}^{\xi}(V_{t})=\tilde \gamma_{t+1}^{\xi}(S_{t+1},V_t)$ for all $V_t\in L^0(\R^d,\cF_t)$. Assume that Assumption \ref{g} and Assumption \ref{CondSuppHemmiCont} hold. Suppose that the cost function $\C_t(s,z)$ is an $\cF_t$-normal integrand and $\C_t$ is either super-additive  or  sub-additive. Then, if $\inf_{z\in S^{d-1}(0,1)}D_t^0(S_t,0,z)>0$,  $\gamma_{t}^{\xi}(V_{t-1})=\tilde \gamma_{t}^{\xi}(S_{t},V_{t-1})$ where $\tilde \gamma_{t}^{\xi}(s,v_{t-1})$ is $\cF_t$-normal integrand.
\end{theo}

\begin{proof}Since $\tilde \gamma_{t+1}^{\xi}(s,v)$ is l.s.c. in $s$, we deduce that $\theta_{t}^\xi(V_{t})=\tilde\theta_{t}^\xi(S_t,V_t)$ by Proposition  \ref{esssup}, for all  $V_t\in L^0(\R^d,\cF_t)$   where, by Assumption \ref{CondSuppHemmiCont},
$$\tilde\theta_{t}^\xi(s,v)=\sup_{z\in\phi_t(S_t)}\tilde \gamma_{t+1}^{\xi}(z,v).$$
As $\phi_t$ is lower hemicontinuous by assumption, we deduce by \cite[Lemma 17.29]{AB} that $\tilde\theta_{t}^\xi(s,v)$ is l.s.c. in $(s,v)$. Therefore, the function   
$$D_t^\xi(s,v_{t-1},v_{t})= \C_{t}(s,(0,v_{t}^{(2)}- v_{t-1}^{(2)})) + \tilde\theta_{t}^\xi(s,v_{t})$$ is l.s.c. in $(s,v_{t-1},v_{t})$ by assumption on $\C_t$. By Lemma \ref{ess omega-wise lemm}, we get that  $\gamma_{t}^{\xi}(V_{t-1})=\tilde \gamma_{t}^{\xi}(S_{t},V_{t-1})$ where $\tilde \gamma_{t}^{\xi}(s,v_{t-1})=\inf_{v_{t}\in \R^d}D_t^\xi(s,v_{t-1},v_{t})$. The next step is to show that $\tilde \gamma_{t}^{\xi}(s,v_{t-1})=\inf_{v_{t}\in \tilde\phi_t(s,v_{t-1})}D_t^\xi(s,v_{t-1},v_{t})$ where $\tilde\phi_t$ is a set-valued upper hemicontinuous function, see Definition \ref{def: upper hemi}, with compact values. We then conclude that $\tilde \gamma_{t}^{\xi}(s,v_{t-1})$ is l.s.c. in $(s,v_{t-1})$ by Proposition \ref{prop:lsc value}.

To obtain $\tilde \phi_t$, first observe that $\gamma_{t}^{\xi}(V_{t-1})\le D_t^\xi(s,v_{t-1},0)$ hence we get that  $\gamma_{t}^{\xi}(V_{t-1})=\tilde \gamma_{t}^{\xi}(S_{t},V_{t-1})$ where $\tilde \gamma_{t}^{\xi}(s,v_{t-1})=\inf_{v_{t}\in K_t(s,v_{t-1})}D_t^\xi(s,v_{t-1},v_{t})$ and 
$$K_t(s,v_{t-1})=\left\{v_{t} \in \R^d:~  D_t^\xi(s,v_{t-1},v_{t})\le D_t^\xi(s,v_{t-1},0) \right\}.$$
Since $\C_T$ is increasing w.r.t. $\R^d_+$, we deduce that $D_t^\xi(s,v_{t-1},v_{t})\ge D_t^0(s,v_{t-1},v_{t})$. Moreover,
$$D_t^0(s,v_{t-1},v_{t})=\C_{t}(s,(0,v_{t}^{(2)}- v_{t-1}^{(2)})) + \theta_{t}^0(s,v_{t})\ge \C_{t}(s,(0,- v_{t-1}^{(2)}))+D_t^0(s,0,v_{t})$$ in the case where $\C_t$ is super-additive  and, if $\C_t$ is sub-additive, we have 
$$D_t^0(s,v_{t-1},v_{t})=\C_{t}(s,(0,v_{t}^{(2)}- v_{t-1}^{(2)})) + \theta_{t}^0(s,v_{t})\ge -\C_{t}(s,(0,v_{t-1}^{(2)}))+D_t^0(s,0,v_{t}).$$
As $\C_t$ is dominated by a continuous function by hypothesis, we get that $D_t^0(s,v_{t-1},v_{t})\ge \tilde h_t(s,v_{t-1})+D_t^0(s,0,v_{t})$ where $\tilde h_t$ is a continuous function. Moreover, by Lemma \ref{Delta-homogD}, if 
$|v_t|\ge 1$, 
\begin{align}\label{eq: D_t^0}
    D_t^0(s,0,v_{t})\ge \delta(|v_t|) D_t^0(s,0,v_{t}/|v_t|)\ge \delta(|v_t|)\ \inf_{z\in S^{d-1}(0,1)}D_t^0(s,0,z).
\end{align}

 By Lemma \ref{ContinuousBoundD}, $|D_t^\xi(s,v_{t-1},0)|\le \hat h_t^\xi(s,v_{t-1})$ for some continuous function $\hat h_t^\xi\ge 0$. Recall that $\inf_{z\in S^{d-1}(0,1)}D_t^0(S_t,0,z)>0$ a.s. by assumption. It follows that   $K_t(s,v_{t-1})\subseteq \tilde \phi_t(s,v_{t-1}):= \bar B_t(0,r_t(s,v_{t-1})+1)$ where
\bean r_t(s,v_{t-1})&:=&\delta^{-1}\left(\frac{\lambda_t(s,v_{t-1})}{i_t(s)}  \right),\\
i_t(s)&:=&\inf_{z\in S^{d-1}(0,1)}D_t^0(s,0,z),\,\lambda_t(s,v_{t-1})=|\tilde h_t(s,v_{t-1})|+\hat h_t^\xi(s,v_{t-1}).
\eean
Since $\lambda_t$ is continuous and $i_t$ is l.s.c. by Proposition \ref{prop:lsc value}, we deduce that $\lambda_t/i_t$ is u.s.c. on the open set $\cO_t:=\{(s,v_{t-1})\in \R^m\times \R^d:~i_t(s,v_{t-1})>0\}$. As $\delta^{-1}$ is continuous and increasing, we finally get that $r_t$ is also u.s.c. in $(s,v_{t-1})\in \cO_t$.  By Lemma \ref{lem: uppercontinuous ball}, we deduce  that the function $\tilde\phi_t$ is upper hemicontinuous in $(s,v_{t-1})\in \cO_t$. Therefore, $\tilde \gamma_{t}^{\xi}(s,v_{t-1})=\inf_{v_{t}\in \tilde\phi_t(s,v_{t-1})}D_t^\xi(s,v_{t-1},v_{t})$ is l.s.c. on  $\cO_t$ by Proposition \ref{prop:lsc value}.  Observe that $(S_t,z)\in\cO_t$ a.s.  for all $z\in S(0,1)$ a.s. under our hypothesis.

Consider the mapping $p_t^\xi(s, v_{t-1}):=\inf_{v_{t}\in \R^d}D_t^\xi(s,v_{t-1},v_{t})$ and its
 l.s.c. regularization   ${\rm cl}p_t^\xi(s, v_{t-1})$. Since $D_t^\xi$ is an $\cF_t$-normal integrand by our assumption, we deduce by \cite[Theorem 14.47]{RW} that ${\rm cl}p_t^\xi(s, v_{t-1})$ is an $\cF_t$-normal integrand. Moreover, we know that on the open set $\cO_t$, $\tilde\gamma_t^\xi(s,v_{t-1})$ is l.s.c. hence coincides with ${\rm cl}p_t^\xi(s, v_{t-1})$ by Lemma \ref{lem: lsc envelope}. Therefore, we deduce that ${\rm cl}p_t^\xi(S_t, v_{t-1}) = \tilde\gamma_t^\xi(S_t,v_{t-1})$ a.s.. The conclusion  follows.

\end{proof}

 The following result asserts that the SAIP condition and the condition $\inf_{z\in S^{d-1}(0,1)}D_t^0(S_t,0,z)>0$, both with AIP, are actually equivalent.

\begin{theo}\label{SAIP-Char} Assume that Assumption \ref{g} holds. Suppose that either Assumption \ref{CondSuppHemmiCont} holds or the cost functions $\C_t(s,z)$ are convex in $z$. Suppose that the cost functions $\C_t(s,z)$ are l.s.c. in $(s,z)$ and $\C_t(s,z)$ are either super-additive  or  sub-additive, for any $t\le T$.  Then, the following statements are equivalent:

\begin{itemize}

\item [1.)] SAIP. \smallskip

\item [2.)] AIP holds and $\inf_{z\in S^{d-1}(0,1)} D_t^0(S_t,0,z)>0$ a.s..
\end{itemize}

\end{theo}

\begin{proof}  Let us show that 1.) implies 2.). Suppose first that Assumption \ref{CondSuppHemmiCont} holds.   As $\gamma^{0} _{T}(Z_T)=\C_T(0,-Z_T^{(2)})$ is an $\cF_T$-normal integrand, we deduce by Proposition \ref{Meas+lsc} that $\theta_{T-1}^{0}(Z_{T-1})$ is an $\cF_{T-1}$-normal integrand. Therefore, the function $D_{T-1}^0(S_{T-1},Z_{T-2},Z_{T-1})$ is an $\cF_{T-1}$-normal integrand. Then by lower semicontinuity on the compact set $S^{d-1}(0,1)$  and by a measurable selection argument, there exists $\hat Z_{T-1}\in L^0(\R^d,\cF_{T-1})$ such that $$\inf_{z\in S^{d-1}(0,1)} D_{T-1}^0(S_{T-1},0,z)=D_{T-1}^0(S_{T-1},0,\hat Z_{T-1}).$$ Moreover, $D_{T-1}^0(S_{T-1},0,\hat Z_{T-1})>0$, i.e. $\inf_{z\in S^{d-1}(0,1)} D_{T-1}^0(S_{T-1},0,z)>0$ under SAIP. By Theorem \ref{PropagTheo}, we deduce that $\gamma_{T-1}^0(S_{T-1},Z_{T-2})$ is an $\cF_{T-1}$-normal integrand. By Proposition \ref{Meas+lsc}, we deduce that $\theta_{T-2}^{0}(Z_{T-2})$ is an $\cF_{T-1}$-normal integrand. Therefore, $D_{T-2}^0(S_{T-2},Z_{T-3},Z_{T-2})$ is an $\cF_{T-2}$-normal integrand and, as previously, we deduce that $\inf_{z\in S^{d-1}(0,1)} D_{T-2}^0(S_{T-2},0,z)>0$ under SAIP.  Then, we may proceed by induction by virtue of Theorem \ref{PropagTheo} and Proposition \ref{Meas+lsc}.  \smallskip

At last, if the cost functions are convex, recall that AEP holds by Lemma \ref{LemmaAEP-AIP}. Then,  it suffices to apply Theorem \ref{Meas+lsc+conv} and  Proposition \ref{Meas+lsc}  to deduce that for fixed $S_t\in L^0(\R^d,\cF_t)$, $D_t^0(S_t,0,z)$ is an $\cF_t$-normal integrand as a function of $z$  so that we may conclude similarly. \smallskip

 Let us show that 2.) implies 1.) Suppose that  $D_t^0(S_t,0,Z_t)=0$ for some $Z_t\in L^0(\R^d\setminus\{0\},\cF_t)$. By Lemma \ref{Delta-homogD}, 
 \bean D_t^0(S_t,0,Z_t)\ge \delta(|Z_t|)D_t^0(S_t,0,Z_t/|Z_t|)\ge \delta(|Z_t|)\inf_{z\in S^{d-1}(0,1)} D_t^0(S_t,0,z)>0.\eean
This yields a contradiction hence the conclusion follows under Assumption  \ref{CondSuppHemmiCont}. 
\end{proof}

We then conclude that, under SAIP, the dynamic programming principle allows to compute $\tilde \gamma_t^\xi$ backwardly so that it is possible to deduce the minimal hedging price $c_0(\xi)=\gamma_0^\xi(0)$.

\begin{theo} \label{LSC+} Assume that Assumption \ref{g} and Assumption \ref{CondSuppHemmiCont} hold. Suppose that the cost functions are normal integrands  and  either super-additive  of  sub-additive. Then, under the condition SAIP, there exists an $\cF_t$-normal integrand $\tilde \gamma_t^\xi$ defined on $\Omega\times\R^m\times\R^m$ such that, for all $V_{t-1}\in L^0(\R^d,\cF_{t-1})$, we have $\gamma_t^\xi(V_{t-1})=\tilde \gamma_t^\xi(S_t,V_{t-1})$. Moreover, the dynamic programming principle (\ref{DPP1}) is computable $\omega$-wise as: 
\bean
\gamma_t^\xi(S_t,V_{t-1})=\inf_{y\in \R}\left(\C_t(S_t,(0,y^{(2)}-V_{t-1}^{(2)}))+\sup_{s\in \phi_t(S_t)}\gamma_{t+1}^\xi(s,y) \right),
\eean
where $\phi_t(S_t)={\rm supp}_{\cF_t}S_{t+1}$. Also, the infimum hedging cost of $\xi$ at any time $t$ is reached, i.e. $\gamma_t^\xi(V_{t-1})$ is a mimimal cost.

\end{theo}
The following proposition shows that the classical Robust No Arbitrage $\text{NA}^{\rm r}$ (\cite[Chapter 3 ]{KS}) used to characterize the super hedging prices in the Kabanov model with proportional transaction costs is stronger than the SAIP condition.

\begin{prop}
Suppose that $\rm int\,\G_t^*\neq\emptyset$ for any $t\leq T$. Then, $\rm NA^{\rm r}$ implies SAIP.
\end{prop}
\begin{proof} Recall that  $\rm NA^{\rm r}$ is equivalent to the existence of a martingale $(K_s)_{s\leq T}$ such that $K_s\in \rm int\, \G_s^*$, \cite[Theorem 3.2.1]{KS}. 
Consider $Z_{T-1}\in L^0(\R^d,\cF_{T-1})$. As $D_{T-1}(0,Z_{T-1})=D_{T-1}(0,(0,Z_{T-1}^{(2)}))$, we may suppose that $Z_{T-1}= (0,Z_{T-1}^{(2)})$.  By the definition of $\C_{u}$, there exists $\tilde g_{u}\in L^0(\G_{u},\cF_{u})$, $u=T-1,T$, such that:
\begin{align*}
    &\C_{T-1}((0,Z_{T-1}^{(2)}))e^1 -  g_{T-1} = (0,Z_{T-1}^{(2)})\\
    &\C_T((0,-Z_{T-1}^{(2)}))e^1 - \tilde g_T = (0,-Z_{T-1}^{(2)}).
\end{align*}
Adding these equalities, we get that $D_{T-1}(0,Z_{T-1})e^1=g_{T-1} +g_T$ for some $g_{T}\in L^0(\G_{T},\cF_{T})$, see (\ref{D1}).  So, we get that  $K_T D_{T-1}(0,Z_{T-1})e^1\ge  K_T  g_{T-1}$ and, taking the generalized conditional expectation w.r.t $\cF_{T-1}$, we deduce that $ K_{T-1}D_{T-1}(0,Z_{T-1})e^1 \geq K_{T-1}g_{T-1}\geq 0.$
Since $ K_{T-1}e^1=K_{T-1}^1>0$,  AIP holds at time $T-1$. Moreover, $g_{T-1}\neq 0$ a.s. as soon as $Z_{T-1}^{(2)}\neq 0$. Since $K_{T-1}\in {\rm int }\, \G_{T-1}^*$, we finally deduce that
\begin{align*}
    K_{T-1}D_{T-1}^0(S_t,0,Z_{T-1})e^1 \geq K_{T-1}g_{T-1} > 0
\end{align*}
as soon as $Z_{T-1}^{(2)}\neq 0$, which means that SAIP holds at time $T-1$. \smallskip

Suppose that we have already shown SAIP for $s\geq t+1$. For a given $Z_t\in L^0(\R^d,\cF_t)$, we consider  $g_t\in L^0(\G_t,\cF_t)$ such that 
\begin{align}\label{eq: C_t g_t}
    \C_t((0,Z_t^{(2)}))e^1 - g_t = (0,Z_t^{(2)}).
\end{align}
Since AIP holds at time $t+1$, by Lemma \ref{LemmaAEP-AIP}, we have $\gamma_{t+1}(Z_t)> -\infty$ under AEP. Since the family $\{D_{t+1}^0(Z_t,Z_{t+1}),Z_{t+1}\in L^0(\R^d,\cF_{t+1})\}$ is directed downward, we deduce the existence of a sequence $Z^n_{t+1}\in L^0(\R^d,\cF_{t+1}),\ n\in\N$ such that
\begin{align*}
    \gamma_{t+1}^0(Z_t) =\essinf_{Z_{t+1}\in L^0(\R^d,\cF_{t+1})}D_{t+1}^0(Z_t,Z_{t+1})= \inf_n D_{t+1}^0(Z_t,Z^n_{t+1}) > -\infty\text{ a.s.}.
\end{align*}

We deduce that, for any $\epsilon > 0$, there exists $Z^{\epsilon}_{t+1}\in L^0(\R^d,\cF_{t+1})$ such that $\gamma_{t+1}^0(Z_t) + \epsilon \geq D_{t+1}^0(Z_t,Z^{\epsilon}_{t+1})$. Proceeding forward with  the induction hypothesis, we construct a sequence $g_{s}^\epsilon\in L^0(\G_s,\cF_s), s\geq t+1$, such that
\begin{align*}
    &(D_t^0(0,Z_t) + \epsilon T)e^1= g_t+ \sum_{s=t+1}^Tg_s^\epsilon.
 \end{align*}
 Therefore, multiplying by $K_T\in \G_T^*$ and then taking the (generalized) conditional expectation knowing $\cF_{T-1}$, we get that
\begin{align*}
    &K_T(D_t^0(0,Z_t) + \epsilon T)e^1\geq\ K_T\left( g_t+\sum_{s=t+1}^{T-1}g_s^\epsilon\right),\\
    &K_{T-1}(D_t^0(0,Z_t) + \epsilon T)e^1 \geq K_{T-1}\left( g_t+\sum_{s=t+1}^{T-1}g_s^\epsilon\right).
\end{align*}

By successive iterations, we finally get that $K_{t}(D_t^0(0,Z_t) + \epsilon T)e^1 \geq K_tg_t$. Since $g_t$ does not depend on $\epsilon$, see its definition in (\ref{eq: C_t g_t}), we deduce as $\epsilon\to 0$, that $K_tD_t^0(0,Z_t)e^1 \geq K_tg_t\ge 0$ and $K_tD_t^0(0,Z_t)e^1>0$ if $g_t\ne 0$ when $Z_t^{(2)}\ne 0$. Therefore, SAIP holds at time $t$ and we may conclude.
\end{proof}

\subsection{The case of fixed transaction costs}\label{FC}
In the case of fixed costs, the cost functions $\C_t$, $t\le T$, are not convex in general. Moreover, $\C_t$ is a priori positively lower homogeneous, i.e. for any $\lambda\ge 1$,  $\C_t(\lambda z)\le \lambda \C_t(z)$. Then,   $\C_t$ does not satisfy the assumptions we impose in this paper. Nevertheless, we shall see in this section that we may also implement the dynamic programming principle under a robust SAIP condition imposed on the enlarged market with only proportional transaction costs.

To do so, recall that for  a l.s.c. function $g$,  the horizon function (see \cite[Section 3.C]{RW}) $g^\infty$ of $g$ is defined as:
\begin{align*}
    g^\infty(y):= \liminf_{\alpha\to\infty}\dfrac{g(\alpha y)}{\alpha}.
\end{align*}
Recall  that $g^\infty$ is positively homogeneous and l.s.c. in $y$.  We then define the \textit{horizon} cost function as 
\bea \label{Cinfty}
    \hat \C_t(s,y)=\C_t^\infty(s,y) = \liminf_{\alpha\to\infty}\dfrac{\C_t(s,\alpha y)}{\alpha}.
\eea

The liquidation value associated to the cost function $\hat \C_t$ is then given by 

\begin{align*}
    \hat \L_t(s,y)= \limsup_{\alpha\to\infty}\dfrac{\L_t(s,\alpha y)}{\alpha}.
\end{align*}

Note that in the case where $\hat \C_t(s,y)= \lim_{\alpha\to\infty}\dfrac{\C_t(s,\alpha y)}{\alpha}$, then    $\hat \L_t= \L_t^{\infty}$. Moreover, if $\hat \C_t$ is subadditive, we deduce that 
$$\hat \G_t(\omega):=\{z:~\hat \L_t(S_t(\omega),z)\ge 0\}$$ is an $\cF_t$-measurable random positive closed cone. We then deduce that the enlarged market defined by the solvency sets $(\hat \G_t)_{t\in [0,T]}$ corresponds to a model with proportional transaction costs, as defined in \cite{KS}[Section 3]. The cash invariance property propagates from $\C_t$ to $\hat \C_t$. In that case, we may verify that $\hat \L_t(s,z)=\max\{ \alpha \in \R:~z-\alpha e_1 \in \hat \G_t\}$ and similarly, we have $\hat \C_t(s,z)=\min\{ \alpha \in \R:~\alpha e_1-z \in \hat \G_t\}$. We then deduce the following:

\begin{lemm} Suppose that $\C_t$ is cash invariant. Then, $\G_t\subseteq \hat \G_t$ if and only if  $\hat \C_t(S_t,z)\le \C_t(S_t,z)$ for any $z$ a.s..
\end{lemm}
\begin{proof}  First suppose that $\G_t\subseteq \hat \G_t$. As $\C_t(S_t,z)e_1-z \in \G_t$, then we get that  $\C_t(S_t,z)e_1-z \in \hat \G_t$. Therefore, we deduce that 
$$\hat \C_t(s,z)=\min\{ \alpha \in \R:~\alpha e_1-z \in \hat \G_t\}\le \C_t(S_t,z).$$
Reciprocally, if $\hat \C_t\le \C_t$, then $\hat \L_t\ge \L_t$ hence $\G_t\subseteq \hat \G_t$.
\end{proof}

Note that in \cite{LT}, such an enlarged model $(\hat \G_t)_{t\in [0,T]}$ is studied and  $\hat \L_t$ is the liquidation value of the closed conic hull ${\rm K}_t$ of $\G_t$, i.e. $\hat \G_t={\rm K}_t$.

\begin{ex}\label{FCM} {\rm The market is composed of one bond whose price is $B_t=1$ and $d-1$ risky assets,  $d\ge 2$, whose prices are described by a family of bid and ask prices and fixed costs $S=((S^{b,i},S^{a,i},c^i))_{i=2,\cdots,d}$. In the following, we denote by $s=((s^{b,i},s^{a,i},c^i))_{i=2,\cdots,d}$ any element of $\R^{3(d-1)}$. We consider the  fixed costs model defined by the following liquidation process:
\bean 
    \L_t(s,y) &:=& y^1 + \sum_{i=2}^d \L_t^i(s^{b,i},s^{a,i},c^i,y^i),\, (s,y) \in \R^{3(d-1)}\times \R^d,\\
    \L_t^i(s^{b,i},s^{a,i},c^i,y^i)&:=& \left(y^is^{b,i} - c_t^i\right)^+1_{y^i >  0} +\left(y^is^{a,i} - c_t^i\right)1_{y^i<0}.
 \eean
 Note that the $(c^i)_{i=2,\cdots,d}$ are interpreted as fixed costs while $(s^{b,i},s^{a,i})_{i=2,\cdots,d}$ are bid and ask prices for the risky assets. We may of course generalize  this model to an order book with several bid and ask prices for each asset, as in Example \ref{OB}. Recall that by definition  $\C_t(s,y)= -\L_t(s,-y)$ and we may verify that $\C_t(s,y)$ is l.s.c. in every $(s,y)$ such that $(c^i)_{i=2,\cdots,d}\in \R_+^{d-1}$. To see it, it suffices to observe that  $\L_t^i(s,y)$ is continuous at each point  $(s,y)$ such that $y\ne 0$. At last, if $y=0$, $\L_t(s,y)=0$ and $\liminf_{r\to s,y\to 0}\L_t(r,y)\le 0$ since $c_t^i\ge 0$. Therefore,  $\L_t^i$ is u.s.c. Moreover,   $\C_t(s,y)$ is subadditive in $y$.  A direct computation yields that $\hat \L_t(s,y) =y^1 + \sum_{i=2}^d \hat \L_t^i(s^{b,i},s^{a,i},y^i)$ where
$$ \hat \L_t^i(s^{b,i},s^{a,i},y^i)= (y^i)^+s^{b,i}- (y^i)^-s^{a,i}.$$

Note that $\hat \L_t=\L_t^{\infty}$ and we have  $\hat \C_t(s,y) =y^1 + \sum_{i=2}^d \hat \C_t^i(s^{b,i},s^{a,i},y^i)$ where 
$$\hat \C_t^i(s^{b,i},s^{a,i},y^i)=(y^i) ^+s^{a,i}-(y^i)^-s^{b,i}.$$
Observe that $\hat \L_t$ and $\hat \C_t$ are continuous in $(s,y)$. Moreover, $\hat \C_t\le \C_t$ and $\hat \C_t$ is  super $\delta$-homogeneous with $\delta(x)=x$.} $\triangle$
\end{ex}

In the following, we adapt the notations of Section \ref{DPPP} to the enlarged model $(\hat \G_t)_{t\in [0,T]}$ as follows:
We set 
$$\hat \gamma_T(S_T,V_{T-1}) = g^1(S_T) + \hat \C_T(S_T,(0,g^{(2)}(S_T)-V_{T-1}^{(2)})),$$ and we  define recursively
\bean
\hat\theta_{t}^\xi(V_{t})&:=& \esssup_{\cF_{t}}\hat\gamma_{t+1}^\xi(V_{t}),\\
\hat D_t^\xi(S_t,V_{t-1},V_{t})&:=& \hat \C_{t}(S_t,(0,V_{t}^{(2)}- V_{t-1}^{(2)})) + \hat\theta_{t}^\xi(S_t,V_{t}).
\eean

\begin{theo} \label{PropagTheoHorizon} Suppose that the enlarged market satisfies $\hat \C_t\le \C_t$,  $\hat \C$ is super $\delta$-homogeneous and is either sub-additive or super-additive. Suppose that  there exists an $\cF_{t+1}$-normal integrand $\tilde \gamma_{t+1}^{\xi}$ defined on $\R^m\times \R^d$ such that   $\gamma_{t+1}^{\xi}(V_{t})=\tilde \gamma_{t+1}^{\xi}(S_{t+1},V_t)$ for all $V_t\in L^0(\R^d,\cF_t)$. Assume that Assumption \ref{g} and Assumption \ref{CondSuppHemmiCont} hold. Suppose that the cost function $\C_t(s,z)$ is an $\cF_t$-normal integrand and $\C_t$ is either super-additive  or  sub-additive. Then, if $\inf_{z\in S^{d-1}(0,1)}\hat D_t^0(S_t,0,z)>0$,  $\gamma_{t}^{\xi}(V_{t-1})=\tilde \gamma_{t}^{\xi}(S_{t},V_{t-1})$ where $\tilde \gamma_{t}^{\xi}(s,v_{t-1})$ is an $\cF_t$-normal integrand.
\end{theo}

\begin{proof}
As $\hat \C_t(x)\leq \C_t(x)$, we deduce by induction that $\hat D_t^0(s,0,v_t) \leq D_t^0(s,0,v_t)$. We adapt the main arguments of  the proof of Theorem \ref{PropagTheo}. Recall that $D_t^0(s,v_{t-1},v_{t})\ge \tilde h_t(s,v_{t-1})+D_t^0(s,0,v_{t})$ where $\tilde h_t$ is a continuous function. By Lemma  \ref{Delta-homogD}, we  have for $|v_t|\ge 1$,
\begin{align*}
    D_t^0(s,0,v_{t})\ge \hat D_t^0(s,0,v_{t})\ge \delta(|v_t|) \hat D_t^0(s,0,v_{t}/|v_t|)\ge \delta(|v_t|) \inf_{z\in S^{d-1}(0,1)}\hat D_t^0(s,0,z).
\end{align*} 
Therefore, we also get that  $\tilde \gamma_{t}^{\xi}(s,v_{t-1})=\inf_{v_{t}\in K_t(s,v_{t-1})}D_t^\xi(s,v_{t-1},v_{t})$ where    $K_t(s,v_{t-1})\subseteq \phi_t(s,v_{t-1}):= \bar B_t(0,r_t(s,v_{t-1})+1)$ and
\bean r_t(s,v_{t-1})&:=&\delta^{-1}\left(\frac{\lambda_t(s,v_{t-1})}{i_t(s)}  \right),\\
i_t(s)&:=&\inf_{z\in S^{d-1}(0,1)}\hat D_t^0(s,0,z),\,\lambda_t(s,v_{t-1})=|\tilde h_t(s,v_{t-1})|+\hat h_t^\xi(s,v_{t-1}).
\eean

Applying Theorem \ref{PropagTheo} by induction to the enlarged market, we deduce that $\hat D_t^0(s,0,z)$ is l.s.c. in $(s,z)$, see the  proof of Theorem \ref{PropagTheo}. We then conclude as in the proof of Theorem \ref{PropagTheo}.

\end{proof}

\begin{rem}{\rm  Recall that the condition $\inf_{z\in S^{d-1}(0,1)}\hat D_t^0(S_t,0,z)>0$ we impose in the theorem above means that SAIP holds for the enlarged market, a priori without fixed cost. Moreover, the other conditions we impose are also satisfied in the fixed costs model of Example \ref{FCM}. } $\triangle$

\end{rem}

\subsection{Computational feasibility under a weaker SAIP no-arbitrage condition}

In this section, we consider a no-arbitrage condition called LAIP, weaker than SAIP, but still sufficient to deduce that the essential infimum in the dynamic programming principle (\ref{DPP}) is a pointwise infimum so that it can be numerically computed.

\begin{lemm}\label{sub-add D}

Suppose that $\C_t$ is sub-additive for any $t\le T$. Then, for any payoff $\xi\in L^0(\R^d,\cF_T)$, the function $D_{t}^\xi$ defined by (\ref{D1}) satisfies the following inequality:
\begin{align*}
    D_{t}^\xi(V_{t-1}+\bar V_{t-1},V_t + \bar V_t) \leq D_t^\xi(V_{t-1},V_t) + D_t^0(\bar V_{t-1},\bar V_t).
\end{align*}
\end{lemm}
\begin{proof}
By definition, with the sub-additivity of $\C_T$, we have:
\bean
    \gamma_T^\xi(V_{T-1}+\bar V_{T-1}) &=& \xi^1 + \C_T((0,\xi^{(2)} - V_{T-1}^{(2)} - \bar V_{T-1}^{(2)})), \\
    &=& \xi^1 + \C_T((0,-V_{T-1}^{(2)})) + \C_T((0,-\bar V_{T-1}^{(2)})),\\
    &\leq& \gamma_T^\xi(V_{T-1}) + \gamma_T^0(\bar V_{T-1}).
\eean
We deduce that $\theta_{T-1}^\xi(V_{T-1}+\bar V_{T-1})\leq \theta_{T-1}^\xi(V_{T-1}) + \theta_{T-1}^0(\bar V_{T-1})$ and, since $D_{T-1}^\xi(V_{T-2},V_{T-1}) = \C_{T-1}((0,V_{T-1}-V_{T-2})) + \theta_T^\xi(V_{T-1})$, we get that:
\begin{align*}
    D_{T-1}^\xi(V_{T-2}+\bar V_{2-1},V_{T-1} + \bar V_{T-1}) \leq D_{T-1}^\xi(V_{T-2},V_{T-1}) + D_{T-1}^0(\bar V_{T-2},\bar V_{T-1}).
\end{align*}
Taking the essential infimum with respect to $V_{T-1}$ and $\bar V_{T-1}$, we get that
\begin{align*}
    \gamma_{T-1}^\xi(V_{T-2}+\bar V_{T-2})\leq \gamma_{T-1}^\xi(V_{T-2}) + \gamma_{T-1}^0(\bar V_{T-2}).
\end{align*}
We may pursue  by induction and conclude.
\end{proof}

We now introduce the LAIP condition. By Proposition \ref{JointlyMeasVersion}, we may suppose that the function $D_t^0(y,z)$ defined by (\ref{D1}) is l.s.c. in $(y,z)$ and it is $\cF_t\otimes \cB(\R^d)\otimes \cB(\R^d)$ measurable w.r.t. $(\omega,y,z)$. Note that, under AIP, the family of random variables $\cN_t:=\left\{Z_t\in L^0(\R^d,\cF_t), Z_t^1=0,\,\ D_t^0(0,Z_t) = 0\right\}$ coincides with $\left\{Z_t\in L^0(\R^d,\cF_t), Z_t^1=0,\,\ D_t^0(0,Z_t) \le 0\right\}$. Therefore, by lower semicontinuity, $\cN_t$ is a  closed subset of $L^0(\R^d,\cF_t)$. Moreover, $\cN_t$ is $\cF_t$- decomposable, see \cite[Section 5.4]{KS}. Therefore, by  \cite[Proposition 5.4.3]{KS}, there exists an $\cF_t$-measurable random set ${\rm N}_t$ such that $\cN_t=L^0({\rm N}_t,\cF_t)$.
\begin{defi}
We say that the  condition LAIP (Linear AIP condition) holds at time $t$ if AIP holds at time $t$  and $\cN_t$ is a linear vector space, or equivalently ${\rm N}_t$ is a.s. a linear subspace of $\R^d$. We say that LAIP holds if LAIP holds at any time.
\end{defi}
Note that if $\cN_t=\{0\}$, then SAIP, AIP and LAIP are equivalent. In general, SAIP implies LAIP. The following result gives a financial interpretation of LAIP. If LAIP holds, the cost to hedge the zero payoff from an initial risky position  $Z_t=V_t^{(2)}\in  L^0(\R^{d-1},\cF_t)$ is zero if and only if the cost is also zero for the position $-Z_t$. This symmetric property is related to the SRN condition of \cite{LV}.

\begin{lemm} Suppose that $\C_t$ is sub-additive and is positively super $\delta$-homogeneous, for any $t\le T$. The following statements are equivalent:

\begin{itemize}
\item [1.)] LAIP holds.\smallskip

\item [2.)] AIP holds and, if $Z_t\in L^0(\R^d,\cF_t)$, then $D_t^0(0,Z_t)=0$ if and only if $D_t^0(0,-Z_t)=0$, $t\le T$.
\end{itemize}
\end{lemm}
\begin{proof}
The implication 1.) $\implies$ 2.) is immediate. Reciprocally, suppose that 2.) holds. Let us show that  $\cN_t$ is stable under addition. We consider $Z_t^1, Z_t^2\in \cN_t$. By Proposition \ref{sub-add D}, we get under AIP that 
$$0\le D_t^0(0,Z_t^1+ Z_t^2)\le D_t^0(0,Z_t^1)+D_t^0(0, Z_t^2)\le 0.$$
We deduce that $Z_t^1+ Z_t^2\in \cN_t$. By induction, we then deduce that for any integer $n$, $n\cN_t\subseteq \cN_t$. Moreover, by Lemma \ref{Delta-homogD}, if $\lambda_t\in L^0((0,1],\cF_t)$, 
$$D_t^0(0, V_{t})=D_t^0(0, \lambda_t(\lambda_t)^{-1}V_{t})\ge \delta((\lambda_t)^{-1})D_t^0(0, \lambda_tV_{t})\ge 0.$$
So $V_t\in \cN_t$ implies that $\lambda_tV_{t}\in \cN_t$ if $\lambda_t\in L^0((0,1],\cF_t)$. Finally, as $\N\cN_t\subseteq \cN_t$, 
$\lambda_tV_{t}\in \cN_t$ for every $\lambda_t\ge 0$. Moreover, $\cN_t$ is symmetric by assumption. The conclusion follows. \end{proof}

In the following, let us consider ${\rm N}_t^\perp:=\{z\in \R^d:~zx=0,\,\forall x\in {\rm N}_t\}$, the random $\cF_t$-measurable linear subspace orthogonal to ${\rm N}_t$. 

\begin{lemm}\label{lem: constant D} Suppose that $\C_t$ is sub-additive and LAIP holds.  Then, for all  $V_{t-1}\in L^0(\R^d,\cF_t)$, there exists  $V_t^2\in L^0({\rm N}_t^\perp,\cF_t)$ such that 
\begin{align*}
    D_t^\xi(V_{t-1},V_t) = D_t^\xi(V_{t-1},V_t^2)\text{ a.s..}
\end{align*}

\end{lemm}
\begin{proof} By a measurable selection argument, it is possible to decompose any $V_t\in L^0(\R^d,\cF_t)$ into $V_t = V_t^1 + V_t^2$, where $V_t^1\in L^0({\rm N}_t,\cF_t)$, $V_t^2\in L^0({\rm N}_t^\perp,\cF_t)$. By Lemma \ref{sub-add D}, we have
\begin{align*}
    D_t^\xi(V_{t-1},V_t) \leq D_t^\xi(V_{t-1}, V_t^2) + D_t^0(0, V_t^1) = D_t^\xi(V_{t-1}, V_t^2).
\end{align*}
On the other hand,  as $V_t^2 = V_t - V_t^1$ and $- V_t^1\in \cN_t$ under LAIP, we also have
\begin{align*}
    D_t^\xi(V_{t-1},V_t^2) \leq D_t^\xi(V_{t-1}, V_t) + D_t^0(0, -V_t^1) = D_t^\xi(V_{t-1}, V_t).
\end{align*}
The conclusion follows.
\end{proof}

In the following, we assume the following condition.

\begin{Assum}\label{assump C_t bound} For any $t\le T$, 
$|\C_t((0,x^{(2)}))| <\bar h_t(x)$, where $\bar h_t$ is a random  function $\bar h_t:(\omega,x)\in\Omega\times\R^d\mapsto \bar h_t(\omega,x)\in\R$ which  is $\cF_t\otimes\cB(\R^d)$-measurable and continuous  a.s.  in $x$.
\end{Assum}

Note that the condition above holds under our initial hypothesis with $\bar h_t(x)=h_t(S_t,x)$ but, here, we do not stress the dependence of $\C_t$ on $S_t$.

\begin{theo} \label{PropagTheo1} Suppose that there exists an $\cF_{t+1}$-normal integrand function $\tilde \gamma_{t+1}^{\xi}$ defined on $\R^d$. Assume that Assumption \ref{assump C_t bound} holds. Suppose that the cost function $\C_t(z)$ is an $\cF_t$-normal integrand and $\C_t$ is sub-additive, positively super $\delta$-homogeneous.  If LAIP holds, then  $\gamma_{t}^{\xi}(V_{t-1})=\tilde \gamma_{t}^{\xi}(V_{t-1})$ where $\tilde \gamma_{t}^{\xi}(v_{t-1})$ is an $\cF_t$-normal integrand
\end{theo}

\begin{proof}
By Lemma \ref{lem: constant D}, we get that
\begin{align*}
    \mathop{\essinf_{\cF_t}}\limits_{V_t\in L^0(\R^d,\cF_t)} D_t^\xi(V_{t-1},V_t) = \mathop{\essinf_{\cF_t}}\limits_{V_t\in L^0({\rm N}_t^\perp,\cF_t)} D_t^\xi(V_{t-1},V_t).
\end{align*}
Since ${\rm N}_t^\perp$ is an $\cF_t$-measurable random closed set, by  Proposition \ref{JointlyMeasVersion} and Lemma \ref{ess omega-wise lemm}, we have
\begin{align*}
    \mathop{\essinf_{\cF_t}}\limits_{V_t\in L^0({\rm N}_t^\perp,\cF_t)} D_t^\xi(V_{t-1},V_t) = \inf_{y\in {\rm N}_t^\perp} D_t^\xi(V_{t-1},y).
\end{align*}
On $\left\{\omega:~{\rm N}_t^\perp(\omega) = \{0\}\right\}\in\cF_t$, we have $\gamma_t^\xi(V_{t-1}) = D_t^\xi(V_{t-1},0)$. On the complementary set, $\left\{{\rm N}_t^\perp \neq \{0\}\right\}\in\cF_t$, under LAIP, we have $\inf_{z\in M_t}D_t^0(0,z) > 0$,
where $M_t = {\rm N}_t^\perp\cap S^{d-1}(0,1)\ne \emptyset$. We now adapt the notations and the main arguments in the proof of Theorem \ref{PropagTheo} with $V_t\in {\rm N}_t^\perp$. In our case, we use Assumption \ref{assump C_t bound} in order to dominate the cost function by a continuous function. By Lemma  \ref{Delta-homogD}, for all $v_t\in {\rm N}_t^\perp$, we may suppose w.l.o.g.  that $v_t^1=0$ and we get that
\begin{align*}
    D_t^0(0,v_{t})\ge \delta(|v_t|) D_t^0(0,v_{t}/|v_t|)\ge \delta(|v_t|) \inf_{z\in M_t}D_t^0(0,z).
\end{align*} 
Moreover, by Assumption \ref{assump C_t bound}, we have:
\begin{align*}
    D_t(v_{t-1},0) = \C_t((0,v_{t-1}^{(2)})) + \theta_t^\xi(0) \leq \bar h_t(v_{t-1}) + \theta_t^\xi(0).
\end{align*}

Therefore, we deduce that  $\tilde \gamma_{t}^{\xi}(v_{t-1})=\inf_{v_{t}\in \phi_t(v_{t-1})}D_t^\xi(v_{t-1},v_{t})$ where $\phi$ is the set-valued mapping     $\phi_t(v_{t-1}):= \bar B_t(0,r_t(v_{t-1})+1)$ and
\bean r_t(v_{t-1})&:=&\delta^{-1}\left(\frac{\lambda_t(v_{t-1})}{i_t}  \right),\\
i_t&:=&\inf_{z\in M_t} D_t^0(0,z),\,\lambda_t(v_{t-1})= \tilde h_t(v_{t-1})+\bar h_t(v_{t-1})+\theta_t^\xi(0).
\eean
 By Corollary \ref{coro: measurability}, $i_t>0$ is $\cF_t$-measurable while $\lambda_t(\omega, v_{t-1})$ is $\cF_t\otimes\cB(\R^d)$-measurable and continuous in $v_{t-1}$. Therefore,  $r_t(\omega, v_{t-1})$ is $\cF_t\otimes\cB(\R^d)$-measurable and continuous in $v_{t-1}$. We deduce that $\bar B_t(0,r_t(v_{t-1}))$ is a continuous set-valued mapping by Corollary \ref{ContBall}. We then conclude  by Proposition \ref{prop:lsc value}.

\end{proof}

Note that the theorem above states that, under LAIP, $\gamma_{t}^{\xi}(V_{t-1})$ is a lower-semicontinuous function of $V_{t-1}$. Therefore, by Lemma \ref{ess omega-wise lemm}, $\gamma_{t}^{\xi}(V_{t-1})$ may be computed pointwise as $\gamma_{t}^{\xi}(V_{t-1})=\inf_{y\in \R^d}\left( \C_t((0,y^{(2)}-V_{t-1}^{(2)}))+\theta_{t}^{\xi}(y) \right).$ Moreover, the infimum is reached so that $\gamma_{t}^{\xi}(V_{t-1})$ is a minimal cost.

\section{Appendix}

\subsection{Normal integrands}

\begin{defi}\label{NI}
Let $\cF$ be a complete $\sigma$-algebra.  We say that the function $(\omega,x)\in \Omega\times\R^k \mapsto f(\omega,x)\in \R$ is  an $\cF$-normal integrand if $f$ is $\cF\otimes\cB(\R^k)$-measurable and lower semi-continuous  in $x$. If $Z\in L^0(\R^k,\cF)$, we use the notation $f(Z):\omega\mapsto f(Z(\omega)) = f(\omega,Z(\omega))$. If $f$ is $\cF\otimes\cB(\R^k)$-measurable then $f(Z)\in L^0(\R^k,\cF)$.
\end{defi}

By \cite[Theorem 14.37]{RW}, we have:

\begin{prop}\label{measurability}
If $f$ is an $\mathcal{F}$-normal integrand, then $\inf_{y\in\R^d}f(\omega,y)$ is $\cF$-measurable and 
$\{(\omega,x)\in\Omega\times\R^d: f(\omega,x) = \inf_{y\in\R^d}f(\omega,y)\}\in\mathcal{F}\otimes\mathcal{B}(\R^d)$ is a measurable closed set.
\end{prop}

\begin{coro} \label{coro: measurability}
For any $\cF$-normal integrand $f:\Omega\times\R^d\to\overline{\R}$ and any $\cF$-measurable random set $A$, let $p(\omega) = \inf_{x\in A} f(\omega,x)$. Then the function $p:\Omega\to\overline{\R}$ is $\cF$-measurable.
\end{coro}
\begin{proof} Let us define $\delta_{A(\omega)}(x)=+\infty$ if $x\notin A(\omega)$ and $\delta_{A(\omega)}(x)=0$ otherwise. Then, the function $g(\omega,x): = f(\omega,x) + \delta_{A(\omega)}(x)$ is an $\cF$-normal integrand since $A$ is closed and $\cF$-measurable. Moreover,  we observe that $p(\omega) = \inf_{x\in A(\omega)} g(\omega,x)$. The conclusion follows from Proposition \ref{measurability}.
\end{proof}

\begin{coro}\label{measurability-compact}
If $f$ is an $\mathcal{F}$-normal integrand, and if $K$ is an $\cF$-measurable set-valued compact set, then  $\inf_{y\in K(\omega)}f(\omega,y)$ is $\cF$-measurable. Moreover, $M(\omega)=\{x\in K(\omega): f(\omega,x) = \inf_{y\in K(\omega)}f(\omega,y)\}\in\mathcal{F}\otimes\mathcal{B}(\R^d)$ is a non-empty $\mathcal{F}$-measurable closed set. In particular, $\inf_{y\in K(\omega)}f(\omega,y)=f(\omega,\hat y)$ where $\hat y\in L^0(M,\cF)\ne \emptyset$.
\end{coro}
{\sl Proof.} It suffices to extend the function $f$ to $\R^d$ by setting $f=+\infty$ on $\R^d\setminus K(\omega)$ so that $f$ is still l.s.c. on $\R^d$. Then, we may apply Proposition \ref{measurability}. Notice that $M(\omega)\ne \emptyset$ a.s. by a compactness argument so that $L^0(M,\cF)\ne \emptyset$ by a measurable selection argument. \fdem \smallskip

In the following, we use the abuse of notation $f(y)=f(\omega,y)$ for  any $f: \Omega\times\R^d\to\overline{\R}$.

\begin{lemm} \label{ess omega-wise lemm}
For any $\cF$-normal integrand $f: \Omega\times\R^d\to\overline{\R}$ and any non-empty $\cF$-measurable closed set $A$, we have:
\begin{align*}
    \essinf_{\cF}\left\{ f(a),\, a\in L^0(A,\cF)\right\} = \inf_{a\in A}f(a) \text{ a.s..}
\end{align*}
\end{lemm}
\begin{proof}
We first prove that
\begin{align*}
    \essinf_{\cF}\left\{ f(a), a\in L^0(A,\cF)\right\} \leq \inf_{a\in A}f(a).
\end{align*}
Recall that $f$ is an $\cF$-normal integrand and $\inf_{a\in A} f(a)$ is $\cF$-measurable by Corollary \ref{coro: measurability}. Therefore, the  set
\begin{align*}
    \{(\omega,a): a\in A(\omega), \inf_{x\in A}f(x) \leq f(a) < \inf_{x\in A}f(x) + 1/n\}
\end{align*}
is $\cF$-measurable and has non-empty $\omega$ sections for each $n\in\N$. By measurable selection argument, we deduce  $a^{n}\in L^0(A,\cF)$ such that 
$$\inf_{a\in A}f(a)\leq f(a^n) < \inf_{a\in A}f(a) + 1/n.$$ This implies that $\lim_nf(a^n) = \inf_{a\in A}f(a)$. Therefore, \begin{align*}
    \inf_{a\in A}f(a) = \inf_nf(a^n) \geq \essinf_{\cF}\left\{ f(a), a\in L^0(A,\cF)\right\}.
\end{align*}

For the reversed inequality,   for each $a\in L^0(A,\cF)$, $f(a) \geq \inf_{a\in A}f(a)$
and, since $\inf_{a\in A}f(a)$ is $\cF$-measurable by Corollary \ref{coro: measurability}, we deduce by definition of the conditional essential infimum that
\begin{align*}
    \essinf_{\cF}\left\{ f(a), a\in L^0(A,\cF)\right\} \geq \inf_{a\in A}f(a) \text{ a.s..}
\end{align*}
\end{proof}
We  recall a result from \cite{BCL} which characterizes a conditional essential supremum as a pointwise supremum on a random set. Let $\cH$ and $\cF$ be two complete sub-$\sigma$-algebras of $\cF_T$ such that $\cH\subseteq\cF$. The conditional support of $X\in L^0(\R^d,\cF)$ with respect to $\cH$ is the smallest $\cH$-graph measurable random set $\text{supp}_{\mathcal{H}}X$ containing the singleton $\{X\}$ a.s., see \cite{BCL}.

\begin{prop}\label{esssup}
Let $h:\Omega\times\R^k\to\R$ be a $\cH\otimes\mathcal{B}(\R^k)$-measurable function which is l.s.c. in $x$. Then, for all $X\in L^0(\R^k,\mathcal{F})$,
\bean
    \esssup_{\mathcal{H}}h(X) = \sup_{x\in\text{supp}_{\mathcal{H}}X}h(x)\quad a.s..
\eean
\end{prop}

\begin{prop} \label{JointlyMeasVersion}
Fix $\xi^1\in L^0(\R,\cF)$ and $d\ge 2$. Let us consider  a random function $f:\Omega\times\R^{d}\to\R$ that satisfies $f(z) = z^1 + f(0,z^{(2)})$, for any $z=(z^1, z^{(2)})\in \R^{d}$. Suppose that $z\mapsto f(z)$ is l.s.c. a.s.. Then, there exists a  $\mathcal{F}_{t-1}\otimes\mathcal{B}(\R^{d-1})$-measurable random function $F_{t-1}^{*}(\omega,y)$ such that, for any $Y_{t-1}\in L^0(\R^{d-1},\cF_{t-1})$, 
$$F_{t-1}^{*}(Y_{t-1}) = \esssup_{\cF_{t-1}}\left(\xi^1+ f(0, Y_{t-1})\right) =: F_{t-1}^{\xi^1,f}(Y_{t-1}),\, {\rm a.s..}$$
Moreover, $F_{t-1}^{*}(\omega,y)$ is l.s.c. in $y$ and if, in addition, $y\in \R^{d-1}\mapsto f(0,y)$ is a.s. convex, then $y\mapsto F_{t-1}^*(\omega,y)$ is a.s. convex.
\end{prop}

{\sl Proof.}  Consider the family of random variables:
\bean
    \Lambda_{t-1} &= \big\{(x_{t-1},y_{t-1})\in L^0(\R^{d},\mathcal{F}_{t-1}): f(-x_{t-1}, y_{t-1}) \leq -\xi^1\big\} \\
    &= \big\{(x_{t-1},y_{t-1})\in L^0(\R^{d},\mathcal{F}_{t-1}): x_{t-1}\geq  F_{t-1}^{\xi^1,f}(y_{t-1})\big\}.
\eean
Notice that $\Lambda_{t-1}$ is closed in $L^0$ since $f$ is l.s.c.. Moreover, $\Lambda_{t-1}$ is $\mathcal{F}_{t-1}$-decomposable, i.e. $g^1_{t-1}1_{A_{t-1}} + g^2_{t-1}1_{A_{t-1}^c}\in \Lambda_{t-1}$ if $g^1_{t-1}$ and $g^2_{t-1}$ belong to $\Lambda_{t-1}$ and  $A_{t-1}\in\mathcal{F}_{t-1}$. By \cite{LM}[Corollary 2.5], there exists an $\cF_{t-1}$-measurable random closed set $\Gamma_{t-1}$ such that  $\Lambda_{t-1} = L^0(\Gamma_{t-1},\mathcal{F}_{t-1})$. Moreover, there is a Castaing representation, i.e. a countable family $(z^n_{t-1})_{n\ge 1}\in \Lambda_{t-1}$ such that  $\Gamma_{t-1}(\omega) = \text{cl}\{z^n_{t-1}(\omega):n\geq 1\}$, $\omega \in \Omega$. We  define 
$$F^*_{t-1}(\omega,y) := \inf\{x\in\R: (x,y) \in\Gamma_{t-1}(\omega)\}.$$ We claim that $
    F^*_{t-1}(\omega,y) = \inf\big\{x\in\mathbb{Q}:(x,y)\in\Gamma_{t-1}(\omega)\big\}.$
Indeed, first we have $F^*_{t-1}(\omega,y) \leq \inf\big\{x\in\mathbb{Q}:(x,y)\in\Gamma_{t-1}(\omega)\big\}$. Moreover, in the case where $F^*_{t-1}(\omega,y)>-\infty$, for every $\epsilon > 0$, there exist $x\in\R$ such that $(x,y)\in\Gamma_{t-1}$ and $F^*_{t-1}(\omega,y)  +\epsilon \geq x$. Choose $\Tilde{x}\in\mathbb{Q}\cap[x,x+\epsilon]$. Observe that  $(\Tilde x,y)\in\Gamma_{t-1}$ as the $y$-sections of $\Lambda_{t-1}$ are upper sets.  We then have:
\bean
    &&F^*_{t-1}(\omega,y)+ 2\epsilon \geq  x +\epsilon \geq \Tilde{x},\\
    && F^*_{t-1}(\omega,y)\geq \Tilde{x} - 2\epsilon \geq \inf\big\{x\in\mathbb{Q}: (x,y)\in\Gamma_{t-1}(\omega)\big\}- 2\epsilon.
\eean
Since $\epsilon$ is arbitrary chosen, we conclude that 
$$F^*_{t-1}(\omega,y)= \inf\big\{x\in\mathbb{Q}: (x,y)\in\Gamma_{t-1}(\omega)\big\}.$$ Notice that when $F^*_{t-1}(\omega,y)=-\infty$, then we may choose $x\to -\infty$ so that we also have $\Tilde{x}\to -\infty$ and we conclude similary. We then deduce that $F_{t-1}^*(\omega,y)$ is $\mathcal{F}_{t-1}\otimes\mathcal{B}(\R^{d-1})$-measurable. Indeed, for every $c < +\infty$, we have: 
    \bean
        \big\{(\omega,y): F^*_{t-1}(\omega,y) \geq c \big\} = \bigcap_{x\in \mathbb{Q}}\big\{(\omega,y): x1_{(\omega,x,y)\in {\rm Graph}\Gamma_{t-1}}\ge c 1_{(\omega,x,y)\in {\rm Graph}\Gamma_{t-1}}\big\}.
    \eean
Since $\Gamma_{t-1}$ is graph-measurable, $\big\{(\omega,y): F^*_{t-1}(\omega,y) \geq c \big\}\in \mathcal{F}_{t-1}\otimes\mathcal{B}(\R^{d-1})$. We then conclude that $F_{t-1}^*$ is $\mathcal{F}_{t-1}\otimes\mathcal{B}(\R^{d-1})$-measurable. Moreover, if $f_t$ is convex, $\Gamma_{t-1}$ is convex a.s. and we deduce that $F_{t-1}^*(\omega,y)$ is convex in $y$ a.s..

Consider a sequence $y^n\in\R^{d-1}$ which converges to $y$ and let us denote $\beta^n:= F_{t-1}^*(\omega,y^n)$. We have $(\beta^n,y^n)\in \Gamma_{t-1}$ if  $\beta^n>-\infty$. If $\inf_n\beta^n = -\infty$, then, up to a subsequence, $F_{t-1}^*(\omega,y) - 1 >\beta^n$ for $n$ large enough, hence $(F_{t-1}^*(\omega,y) - 1,y^n)\in\Gamma_{t-1}(\omega)$ since the $y^n$-sections of $\Gamma_{t-1}$ are upper sets. As $n\to\infty$, we deduce that $(F_{t-1}^*(\omega,y) - 1,y)\in\Gamma_{t-1}(\omega)$, which contradicts the definition of $F_{t-1}^*$. Moreover it is trivial that $F_{t-1}^*(\omega,y)\leq\liminf_n\beta^n$ if $\liminf_n\beta^n=\infty$. Otherwise, $\beta^\infty:= \liminf_n\beta^n <\infty$ and $(\beta^\infty,y)\in\Gamma_{t-1}$ since $\Gamma_{t-1}$ is closed. It follows that $F_{t-1}^*(\omega,y) \leq\beta^\infty = \liminf_n\beta^n$ by the definition of $F_{t-1}^*$. We conclude that $F_{t-1}^*(\omega,x)$ is l.s.c. in $x$.

We show that $F_{t-1}^{\xi^1,f}(Y_{t-1}) = F^*_{t-1}(Y_{t-1})$ a.s. for all $Y_{t-1}\in L^0(\R^{d-1},\mathcal{F}_{t-1})$. We first restrict $\Omega$ to the $\cF_{t-1}$-measurable set $\{\omega:~\Gamma_{t-1}(\omega)\ne \emptyset\}$. We may then consider a measurable selection $(\tilde x_{t-1},\tilde y_{t-1})\in \Gamma_{t-1}\ne \emptyset$ a.s.. By definition, we have $\tilde x_{t-1} \geq F^*_{t-1}(\tilde y_{t-1})$. We deduce that $F^*_{t-1}(\tilde y_{t-1}) < \infty$ a.s.. We define:
\bean
    \widehat Y_{t-1} = \tilde y_{t-1}1_{F^*_{t-1}(Y_{t-1}) = \infty} + Y_{t-1}1_{F^*_{t-1}(Y_{t-1}) < \infty}.
\eean
Then:
\bean
    F^*_{t-1}(\widehat Y_{t-1}) = F^*_{t-1}(\tilde y_{t-1})1_{F^*_{t-1}(Y_{t-1}) = \infty} + F^*_{t-1}(Y_{t-1})1_{F^*_{t-1}(Y_{t-1}) < \infty}.
\eean
Observe that on the set $\{F^*_{t-1}(Y_{t-1}) < \infty\}$, $(F^*_{t-1}(\widehat Y_{t-1}),\widehat Y_{t-1})\in\Gamma_{t-1}$ a.s. since $\Gamma_{t-1}$ is closed. Therefore,  $(F^*_{t-1}(\widehat Y_{t-1}),\widehat Y_{t-1})\in\Lambda_{t-1}=L^0(\Gamma_{t-1},\cF_{t-1})$ and we deduce that $F^*_{t-1}(\widehat Y_{t-1})\geq F_{t-1}^{\xi^1,f}( \widehat Y_{t-1})$ a.s.. We conclude that on the set $\{F^*_{t-1}(Y_{t-1}) <\infty\}$, $F^*_{t-1}(Y_{t-1}) \geq F_{t-1}^{\xi^1,f}(Y_{t-1})$ while the inequality is trivial on the complementary set. On the other hand, let us define
\bean
    \widehat X_{t-1} &=& F_{t-1}^{\xi^1,f}(Y_{t-1})1_{F_{t-1}^{\xi^1,f}(Y_{t-1})<\infty} + F_{t-1}^{\xi^1,f}(\tilde y_{t-1})1_{F_{t-1}^{\xi^1,f}(Y_{t-1})=\infty},\\
    \widehat Y_{t-1}&=& Y_{t-1}1_{F_{t-1}^{\xi^1,f}(Y_{t-1})<\infty} +\tilde y_{t-1}1_{F_{t-1}^{\xi^1,f}(Y_{t-1})=\infty}.
 \eean

Observe that  $( \widehat  X_{t-1},  \widehat Y_{t-1})\in \Lambda_{t-1}$  hence  $F^*_{t-1}(\widehat Y_{t-1})\leq \widehat X_{t-1}$ by definition of $F^*_{t-1}$. Then, $F^*_{t-1}(Y_{t-1}) \leq \widehat X_{t-1}=F_{t-1}^{\xi^1,f}(Y_{t-1})$ on $\{F_{t-1}^{\xi^1,f}(Y_{t-1})<\infty\}$. The inequality is trivial on the complementary set so that we may conclude. \smallskip

On the set $\{\omega:~\Gamma_{t-1}(\omega)= \emptyset\}$, we have $F^*_{t-1}(Y_{t-1})=+\infty$. Moreover, if $F_{t-1}^{\xi^1,f}(Y_{t-1})<\infty$, we deduce that $(F_{t-1}^{\xi^1,f}(Y_{t-1}),Y_{t-1})\in \Gamma_{t-1}=\emptyset$ since $\xi^1+f(0,Y_{t-1})\le F_{t-1}^{\xi^1,f}(Y_{t-1})$. This is a contradiction hence $F_{t-1}^{\xi^1,f}(Y_{t-1})=+\infty$ and the conclusion follows. \fdem \smallskip

\begin{lemm}\label{CountableRepGamma} 

Suppose that Assumption \ref{assumption support} holds and consider an $\cF_{t-1}$-normal integrand $\gamma_t: (\omega,s,y):\Omega\times\R^m\times\R^d\mapsto \gamma_t(\omega,s,y)$.  Then, for any $V_{t-1}\in L^0(\R^d,\cF_{t-1})$, we have:
\begin{align*}
    \esssup_{\cF_{t-1}}\gamma_{t}(S_t,V_{t-1}) &= \sup_{s\in{\rm supp}_{\cF_{t-1}}S_t}\gamma_t(s,V_{t-1})= \sup_{m\ge 1}\gamma_t(\alpha_{t-1}^m(S_{t-1}),V_{t-1}).
\end{align*}
\end{lemm}

{\sl Proof.}
As $(\omega,s)\mapsto \gamma_{t}(\omega,s,V_{t-1}(\omega))$ is an $\cF_{t-1}$-normal integrand under our assumptions,  the first equality holds by Proposition \ref{esssup}. It remains to observe that, if $s\in \text{supp}_{\cF_{t-1}}S_t$, then $s=\lim_m \alpha_{t-1}^m(S_{t-1})$ for a subsequence and, by lower semicontinuity, we deduce that $$\gamma_{t}(s,V_{t-1})\le \liminf_m \gamma_{t}^\xi(\alpha_{t-1}^m(S_{t-1})),V_{t-1})\le  \sup_{m\ge 1} \gamma_{t}^\xi(\alpha_{t-1}^m(S_{t-1})),V_{t-1}).$$ It follows that $\sup_{s\in\text{supp}_{\cF_{t-1}}S_t}\gamma_t(s,V_{t-1})\le \sup_{m\ge 1}\gamma_t(\alpha_{t-1}^m(S_{t-1}),V_{t-1})$ and, finally, the equality holds. \fdem

\subsection{Continuous set-valued functions}

For two topological vector spaces $X,Y$, consider a set-valued function $\phi: X \twoheadrightarrow Y$. We recall the definition of hemicontinuous set-valued mappings as formulated  in \cite{AB}.
\begin{defi}\label{def: lower hemi}
We say that $\phi$ is \textbf{lower hemicontinuous} at $x$ if for every open set $U\subset Y$ such that $\phi(x)\cap U\neq\emptyset$, there exits a neighborhood $V$ of $x$ such that $z\in V$ implies $\phi(x)\cap U \ne\emptyset$.
\end{defi}

\begin{defi}\label{def: upper hemi}
We say that $\phi$ is \textbf{upper hemicontinuous} at $x$ if for every open set $U\subset Y$ such that $\phi(x)\subseteq U$, there is a neighborhood $V$ of $x$ such that $z\in V$ implies $\phi(z)\subset U$.
\end{defi}
\begin{defi}\label{def: continuous}
We say that $\phi$ is \textbf{continuous} at $x$ if it is both upper and lower hemicontinuous at $x$. It is continuous if it is continuous at any point.
\end{defi}

\begin{lemm}\label{lem: uppercontinuous ball}
    Let $f:\R^k\to\R_+$ be an upper semicontinuous function. Then,  the mapping  $x\mapsto \bar{B}(0,f(x))$ is upper hemicontinuous  in the sense of definition \ref{def: upper hemi}.
\end{lemm}
\begin{proof}
The upper hemicontinuity is simple to check. Indeed, consider  an open set in $U\subseteq \R^k$, such that   $\phi(x)=\bar{B}(0,f(x))\subset U$. We may suppose that $U$ is bounded w.l.o.g. and  we deduce $\epsilon>0$ such that $\bar{B}(0,f(x)+\epsilon)\subset U$. By upper semicontinuity, there exists an open set $V$ containing $x$ such that $z\in V$ implies $f(z)\le f(x)+\epsilon$ hence $\phi(z)\subseteq U$. \end{proof}

\begin{lemm}\label{lem: lowercontinuous ball}
    Let $f:\R^k\to\R_+$ be a lower semicontinuous function. Then,  the mapping  $x\mapsto \bar{B}(0,f(x))$ is lower hemicontinuous  in the sense of definition \ref{def: lower hemi}.
\end{lemm}
\begin{proof}
For any ball $B(y,r)\in\R^k$, we have $\bar{B}(0,f(x))\cap B(y,r)\neq\emptyset$ if and only if $f(x) + r > |y|$. We also have $f(x)-\epsilon + r > |y|$ for some small $\epsilon>0$. As $f$ is l.s.c., we deduce that $f(z)\ge f(x)-\epsilon$ for every $z$ in some neighborhood $V$ of $x$. This implies that $f(z) + r > |y|$, i.e. $\bar{B}(0,f(x))\cap B(y,r)\neq\emptyset$ for every $z\in V$.  The conclusion follows.
\end{proof}

\begin{coro}\label{ContBall}
 Let $f:\R^k\to\R_+$ be a continuous function. Then,  the mapping  $x\mapsto \bar{B}(0,f(x))$ is continuous  in the sense of Definition \ref{def: continuous}.
\end{coro}

\begin{lemm} \label{lower hemicontinuous support}
 Consider the set-valued mapping $\alpha:\R^{m}\twoheadrightarrow \R^{m}$ defined by $\alpha(s) =  {\rm cl}\{\alpha^m(s),m\in\N\}$ where $(\alpha^m)_{m\ge 1}$ are continuous functions. Then, $\alpha$ is  lower hemicontinuous.\end{lemm}
\begin{proof}
Consider $\omega\in \Omega$ and  some open set $U\in\R^{d}$. We have $\alpha_t(\omega, z)\cap U\neq\emptyset$ if and only if there is  $m\in\N$ such that $\alpha_t^m(\omega,z)\in U$. Since  $\alpha_t^m(\omega,.)$ is continuous, we deduce that there exists an open neighborhood $V$ of $z$ such that $\alpha_t^m(\omega,x)\in U$ for any $x\in V$. The conclusion follows.
\end{proof}

We recall a result from \cite{AB}[Theorem 17.31].
\begin{prop}\label{prop:continuous value}
Let $\phi: \R^k\twoheadrightarrow\R^m$ be a continuous set-valued mapping with nonempty compact values and suppose that  $f:\R^k\times\R^m\to\R$ is continuous. Then, the function $m(x) = \inf_{y\in\phi(x)}f(x,y)$ and the function $M(x) = \sup_{y\in\phi(x)}f(x,y)$  are continuous.
\end{prop}

\begin{prop}\label{prop:lsc value}
Let $\phi: \R^k\twoheadrightarrow\R^m$ be an upper hemicontinous set-valued mapping with nonempty compact values and suppose that  $f:\R^k\times\R^m\to\R$ is lower semicontinuous.  Then, the function $m(x) = \inf_{y\in\phi(x)}f(x,y)$ is l.s.c.
\end{prop}
{\sl Proof.} We have $m(x)=-\sup_{y\in\phi(x)}g(x,y)$ where $g=-f$ is upper semicontinuous. By  \cite{AB}[Lemma 17.30], the mapping $x\mapsto \sup_{y\in\phi(x)}g(x,y)$ is upper semicontinuous hence $m$ is l.s.c. \fdem

\begin{lemm}\label{lem: lsc envelope}
 Let ${\rm cl}f$ be the l.s.c. regularization of the function $f:\R^k\to\R$ ( i.e. the greatest l.s.c. function dominated by $f$). Suppose that $f$ is l.s.c. on some open set $O\subset\R^k$, then $f(\bar x) = {\rm cl}f(\bar x)$ for any $\bar x\in O$.
\end{lemm}
\begin{proof}
We define $g(x) := {\rm cl}f(x)1_{O^c}(x) + f(x)1_{O}(x)$. As ${\rm cl}f\le f$ and $O$ is open, we deduce that  $g$ is l.s.c. and $g\leq f$. By definition of ${\rm cl}f$, we have $g\le {\rm cl}f$. This implies that $f(\bar x) \le {\rm cl}f(\bar x) \le f(\bar x)$ for any $\bar x\in O$. The conclusion follows.
\end{proof}

\subsection{Auxiliary  results}

\begin{lemm}\label{CountSupp} Suppose that there is a family of  $\cF_{t-1}$-measurable random variables  $(\alpha_{t-1}^m)_{m\ge 1}$ such that $S_t\in \{\alpha_{t-1}^m:~m\ge 1\}$ a.s. and suppose that $P(S_t=\alpha_{t-1}^m|\cF_{t-1})>0$ a.s. for all $m\ge 1$. Then, for any $\cF_{t-1}$-measurable random function $f:\Omega\times\R^d\to\R$,
\begin{align*}
    \esssup_{\cF_{t-1}}f(S_t) = \sup_{m\geq 1}f(\alpha_{t-1}^m).
\end{align*}
\end{lemm}
\begin{proof}
It is clear that $\esssup_{\cF_{t-1}}f(S_t) \leq \sup_{m\geq 1}f(\alpha_{t-1}^m)\text{ a.s.}$ since $S_t$ belongs to $\{\alpha_{t-1}^m:~m\ge 1\}$ and $\sup_{m\geq 1}f(\alpha_{t-1}^m)$ is $\cF_{t-1}$-measurable by assumption. On the other hand, consider $\Gamma_t^m:= \{S_t\in \alpha_{t-1}^m\}\in \cF_t$. We have:
\begin{align*}
    \esssup_{\cF_{t-1}}f(S_t)1_{\Gamma_t^m}\ge f(S_t)1_{\Gamma_t^m}  \geq f(\alpha_{t-1}^m)1_{\Gamma_t^m}\text{ a.s..}
\end{align*}
Taking the conditional expectation, we get that
\bean
    &&E(\esssup_{\cF_{t-1}}f(S_t)1_{\Gamma_t^m}|\cF_{t-1}) \geq E(f(\alpha_{t-1}^m)1_{\Gamma_t^m}|\cF_{t-1})\text{ a.s.},\\
    &&\esssup_{\cF_{t-1}}f(S_t)P(\Gamma_t^m|\cF_{t-1})) \geq f(\alpha_{t-1}^m)P(\Gamma_t^m|\cF_{t-1})) \text{ a.s..}
\eean
As $P(\Gamma_t^m|\cF_{t-1}))>0$ by assumption, we get that $\esssup_{\cF_{t-1}}f(S_t)\geq f(\alpha_{t-1}^m)$ a.s. for any $m\geq 1$ so that the reverse inequality holds.
\end{proof}

\begin{lemm} \label{Delta-homogD} Let $D^{0}$ given by (\ref{D1}) with $\xi=0$. Suppose that $\C$ is positively super $\delta$-homogeneous. For any $t\le T$, and any $\lambda_t\in L^0([1,\infty),\cF_t)$, we have $D_t^0(\lambda_t V_{t-1},\lambda_t V_{t})\ge \delta(\lambda_t)D_t^0(V_{t-1}, V_{t})$ and $\gamma^0_t(\lambda_t V_{t-1})\ge \delta(\lambda_t)\gamma^0_t(V_{t-1}) $ for all $(V_{t-1}, V_{t})\in L^0(\R^d,\cF_t)\times L^0(\R^d,\cF_t)$.
\end{lemm}
\begin{proof}  For $t=T$, we have by assumption:
$$\gamma^0_T(\lambda_T V_{T-1})=\C_T((0,-\lambda_T V_{T-1}^{(2)})\ge \delta(\lambda_T)\C_T((0,- V_{T-1}^{(2)})= \delta(\lambda_T)\gamma^0_T( V_{T-1}).$$ 
We deduce that 
\bean \theta_{T-1}^0(\lambda_{T-1} V_{T-1})&=&\esssup_{\cF_{T-1}}\gamma^0_T(\lambda_{T-1}V_{T-1}),\\&\ge&  \delta(\lambda_{T-1})\esssup_{\cF_{T-1}}\gamma^0_T( V_{T-1}),\\
&\ge& \delta(\lambda_{T-1})\theta_{T-1}^0( V_{T-1}). \eean
As we also have $$\C_{T-1}((0,\lambda_{T-1} V_{T-1}^{(2)}-\lambda_{T-1} V_{T-2}^{(2)}))\ge \delta(\lambda_{T-1})\C_{T-1}((0,V_{T-1}^{(2)}- V_{T-2}^{(2)})),$$
we deduce that 
\bean D_{T-1}(\lambda_{T-1} V_{T-2},\lambda_{T-1} V_{T-1})&=&\C_{T-1}((0,\lambda_{T-1} V_{T-1}^{(2)}-\lambda_{T-1} V_{T-2}^{(2)}))+\theta_{T-1}^0(\lambda_{T-1} V_{T-1}),\\
&\ge& \delta(\lambda_{T-1})\C_{T-1}((0,V_{T-1}^{(2)}- V_{T-2}^{(2)}))+\delta(\lambda_{T-1})\theta_{T-1}^0( V_{T-1}),\\
&\ge& \delta(\lambda_{T-1}) D_{T-1}(V_{T-2}, V_{T-1}).
\eean
Therefore, as $\lambda_{T-1}\ge 1$,
\bean \gamma^0_{T-1}(\lambda_{T-1} V_{T-2})&=&\essinf_{V_{T-1}\in L^0(\R^d,\cF_{T-1})}D_{T-1}(\lambda_{T-1} V_{T-2},\lambda_{T-1} V_{T-1}),\\
&\ge& \delta(\lambda_{T-1})\essinf_{V_{T-1}\in L^0(\R^d,\cF_{T-1})}D_{T-1}(V_{T-2},V_{T-1}),\\
&\ge& \delta(\lambda_{T-1})\gamma^0_{T-1}( V_{T-2}).
\eean
We then conclude by induction.
\end{proof}

\begin{lemm} \label{ContinuousBoundD} Suppose that Assumption \ref{g} and Assumption \ref{CondSuppHemmiCont}  hold. For every $t\le T$, there exists  a  continuous function $\hat h_t\ge 0$ such that the function $D_t^\xi$ given by (\ref{D2}) satisfies $|D_t^\xi(s,v_{t-1},0)|\le \hat h_t^{\xi}(s,v_{t-1})$.
\end{lemm}
{\sl Proof.} Recall that $\gamma_T^{\xi}(V_T)=g^1(S_T)+\C_T(S_T,(0,g^2(S_T)-V_T^{(2)}))$. By assumption on $\C_T$ and $g$, we deduce that $\gamma_T^{\xi}(V_T)\le f_T(S_T,V_T)$ where $f_T$ is continuous. Therefore, by Proposition \ref{esssup},
\bean \theta_{T-1}^{\xi}(V_{T-1})&=&\esssup_{\cF_{T-1}}\gamma_T^{\xi}(V_{T-1})\le \esssup_{\cF_{T-1}} f_T(S_T,V_{T-1}),\\
&\le &\sup_{z\in {\rm supp}_{\cF_{T-1}}S_{T}}f_T(z,V_{T-1})\le \sup_{z\in \bar B(0,R_{T-1}(S_{T-1}))}f_T(z,V_{T-1}).
\eean
As $R_{T-1}$ is continuous, we deduce by Corollary \ref{ContBall} and Proposition \ref{prop:continuous value} that $\bar\theta_{T-1}^{\xi}(S_{T-1},V_{T-1})=\sup_{z\in \bar B(0,R_{T-1}(S_{T-1}))}f_T(z,V_{T-1})$ is a continuous function in $(S_{T-1},V_{T-1})$. Recall that $\C_{T-1}(S_{T-1},(0,-V_{T-1}^{(2)})\le h_{T-1}(S_{T-1},V_{T-1})$ where $h_{T-1}$ is continuous. As  
$$D_{T-1}^\xi(S_{T-1},V_{T-1},0)=\C_{T-1}(S_{T-1},(0,-V_{T-1}^{(2)})+\theta_{T-1}^{\xi}(V_{T-1}),$$ we deduce that 
$D_{T-1}^\xi(S_{T-1},V_{T-1},0)\le  \hat h^{\xi}_{T-1}(S_{T-1},V_{T-1})$ where $\hat h^{\xi}_{T-1}$ is given by  $\hat h^{\xi}_{T-1}(S_{T-1},V_{T-1})=\bar\theta_{T-1}^{\xi}(S_{T-1},V_{T-1})+ h_{T-1}(S_{T-1},V_{T-1})$, i.e. $\hat h^{\xi}_{T-1}$ is continuous.
Since $\gamma_{T-1}^{\xi}(S_{T-1},V_{T-1})\le D_{T-1}^\xi(S_{T-1},V_{T-1},0)$, we deduce that $\gamma_{T-1}^{\xi}(S_{T-1},V_{T-1})\le \hat h^{\xi}_{T-1}(S_{T-1},V_{T-1})=f_{T-1}(S_{T-1},V_{T-1})$ and we may proceed by induction to conclude. \fdem \smallskip

Following the same arguments, we also deduce the following:

\begin{lemm} \label{ContinuousLowerBoundGamma} Suppose that Assumption \ref{g} and Assumption \ref{CondSuppHemmiCont}  hold. For every $t\le T$, there exists  a  continuous function $\bar h_t$ such that $\gamma_t^{\xi}(V_t)\ge \bar h_t(S_t,V_t)$.
\end{lemm}

\bigskip
{\small On behalf of the organizing committee, Emmanuel Lépinette thanks the "Fondation Natixis pour la recherche et l'innovation" for their financial contribution to the Bachelier Colloquium (ETH Zurich (main sponsor) and Franche-Comté university), Metabief, France. }

\end{document}